\documentclass[a4paper,12pt,reqno]{amsart}
\usepackage{amssymb,amsthm}
\usepackage{amsmath}
\usepackage{mathrsfs}
\usepackage{ifthen}
\usepackage{relsize}
\usepackage{caption}
\usepackage{accents}
\usepackage{graphicx}   
\usepackage{listings}
\usepackage{color} 
\usepackage{xcolor}
\usepackage{subcaption}

\usepackage{kantlipsum}
\definecolor{codegreen}{rgb}{0,0.6,0}
\definecolor{codegray}{rgb}{0.5,0.5,0.5}
\definecolor{codepurple}{rgb}{0.58,0,0.82}
\definecolor{backcolour}{rgb}{0.95,0.95,0.92}
\lstdefinestyle{mystyle}{
    backgroundcolor=\color{backcolour},   
    commentstyle=\color{codegreen},
    keywordstyle=\color{magenta},
    numberstyle=\tiny\color{codegray},
    stringstyle=\color{codepurple},
    basicstyle=\ttfamily\footnotesize,
    breakatwhitespace=false,         
    breaklines=true,                 
    captionpos=b,                    
    keepspaces=true,                 
    numbers=left,                    
    numbersep=6pt,                  
    showspaces=false,                
    showstringspaces=false,
    showtabs=false,                  
    tabsize=4
}

\usepackage{multirow}  
\usepackage{float}
\usepackage{tcolorbox}
\usepackage[numbered]{matlab-prettifier}
\usepackage{diagbox}
\usepackage{cleveref}
\usepackage{caption}
\usepackage{subcaption}
\usepackage{textcomp,multirow}
\theoremstyle{plain}
\newtheorem{thm}{Theorem}[section]
\newtheorem*{thm*}{Theorem}
\newtheorem{algo}{Sturm's Algorithm}[section]

\numberwithin{equation}{section}
\newtheorem{cor}{Corollary}[section]
\newtheorem{lem}{Lemma}[section]

\theoremstyle{definition}
\newcounter {own}
\def\theown {\thesection  .\arabic{own}}

{\qed\bigskip}

\graphicspath{ {./images/} }
\newcommand{\sgn}{\operatorname{sgn}}

\DeclareMathOperator*{\esssup}{ess\,sup}
\DeclareMathOperator*{\essinf}{ess\,inf}
\DeclareMathOperator*{\supp}{supp}
\newcounter{alphabet}

\newcommand{\ds}{\displaystyle}

\newcounter{minutes}
\divide\time by 60
\newcounter{hours}
\multiply\time by 60 \addtocounter{minutes}{-\time}

\begin{document}
\bibliographystyle{amsplain}
\title{On Gabor frames generated by  B-splines, totally positive functions, and Hermite functions}

\thanks{
File:~AntonyRiya5.tex,
          printed: 2023-04-20,
          \thehours.\ifnum\theminutes<10{0}\fi\theminutes}

\author{Riya Ghosh}
\author{A. Antony Selvan}

\address{Riya Ghosh, Indian Institute of Technology Dhanbad, Dhanbad-826 004, India.}
\email{riya74012@gmail.com}
\address{A. Antony Selvan, Indian Institute of Technology Dhanbad, Dhanbad-826 004, India.}
\email{antonyaans@gmail.com}

\maketitle
\pagestyle{myheadings}
\markboth{Riya Ghosh and A. Antony Selvan}{On Gabor frames generated by B-splines, totally positive functions, and Hermite functions}
\begin{abstract}
The frame set of a window $\phi\in L^2(\mathbb{R})$ is the subset of all lattice parameters
$(\alpha, \beta)\in \mathbb{R}^2_+$ such that $\mathcal{G}(\phi,\alpha,\beta)=\{e^{2\pi i\beta m\cdot}\phi(\cdot-\alpha k) : k, m\in\mathbb{Z}\}$ forms a frame for $L^2(\mathbb{R})$. In this paper, we investigate the frame set of B-splines, totally positive functions, and Hermite functions. We derive a sufficient condition for Gabor frames using the connection between sampling theory in shift-invariant spaces and Gabor analysis. As a consequence, we obtain a new frame region belonging to the frame set of B-splines and Hermite functions. For a class of functions that includes certain totally positive functions, we prove that for any choice of lattice parameters $\alpha, \beta>0$ with $\alpha\beta<1,$ there exists a $\gamma>0$ depending on $\alpha\beta$ such that $\mathcal{G}(\phi(\gamma\cdot),\alpha,\beta)$ forms a frame for $L^2(\mathbb{R})$.
\end{abstract}
\section{Introduction}
Given a nonzero window function $\phi\in L^2(\mathbb{R})$ and lattice parameters $\alpha$, $\beta>0$, the Gabor system  $\mathcal{G}(\phi,\alpha, \beta)=\{e^{2\pi i\beta m\cdot}\phi(\cdot-\alpha k): k, m\in\mathbb{Z}\}$ is called a (Gabor) frame for $L^2(\mathbb{R})$ if there exist two positive constants $A$, $B$ such that 
\begin{equation}\label{Gaborframe}
A\|f\|^2\leq\displaystyle\sum_{n\in\mathbb{Z}}|\langle f,e^{2\pi i\beta m\cdot}\phi(\cdot-\alpha k)\rangle|^2\leq B\|f\|^2,
\end{equation}
for every $f\in L^2(\mathbb{R})$. The constants $A$ and $B$ are called frame bounds. One of the fundamental problems in Gabor analysis is to determine the values of $\alpha, \beta>0$ such that $\mathcal{G}(\phi,\alpha, \beta)$ is a frame for $L^2(\mathbb{R})$. The set of all such lattice parameters is referred to as the frame set of $\phi$ and is given by
$$\mathcal{F}(\phi)= \left\{(\alpha, \beta) \in\mathbb{R}^2_+ : \mathcal{G}(\phi, \alpha,\beta)~ \text{is a frame}\right\} .$$
Throughout this paper, we use the following version of the Fourier transform: 
$$\widehat f(w)= \int_{-\infty}^{\infty} f(x) e^{-2\pi \mathrm{i}wx}dx,~ w\in\mathbb{R}.$$
It follows from the definition of Fourier transform that $\mathcal{F}(\phi)=\mathcal{F}(\widehat{\phi})$. Feichtinger and Kaiblinger \cite{HGF1} proved that 
$\mathcal{F}(\phi)$ is an open subset of $\mathbb{R}^2_+$ for a window $\phi$ in Feichtinger algebra. The fundamental density theorem asserts that 
$$\mathcal{F}(\phi)\subseteq \{(\alpha, \beta) \in\mathbb{R}^2_+: \alpha\beta\le 1\}$$ 
(see \cite{time, hedtg}). In addition, if $\phi$ is in Feichtinger algebra, Balian-Low theorem states that $\mathcal{F}(\phi)\subseteq \{(\alpha, \beta) \in\mathbb{R}^2_+: \alpha\beta< 1\}$ \cite{balian, rdbalian}. For a more comprehensive discussion on Gabor analysis, we refer to \cite{ole, time}. The frame set is completely characterized only for a few windows: the Gaussian $e^{-\pi x^2}$ \cite{fblyu, dtsibf2}, the hyperbolic secant \cite{hsyg}, the two-sided exponential  $e^{-|x|}$ \cite{tgfcd}, the one-sided exponential $e^{-x}\chi_{[0,\infty)}(x)$ \cite{whfb},  the characteristic function $\chi_{[0,c)},~c>0$ \cite{abc, when}, the totally positive functions of finite type $\ge 2$ or of Gaussian type \cite{ duke, stsis}, and the Herglotz functions \cite{gfrf}.

The construction of Gabor frames is closely connected to the sampling problem in shift-invariant spaces \cite{duke}. To explain this connection more precisely, let us introduce the shift-invariant space of a generator $\phi$ defined by 
$$V_h(\phi):=\left\{f\in L^2(\mathbb{R}): f(\cdot)=\sum\limits_{k\in\mathbb{Z}}d_k\phi(\cdot-hk)~\text{for some}~ (d_k)\in \ell^2(\mathbb{Z})\right\},~h>0.$$
Recall that $\phi$ is said to be stable generator for $V_h(\phi)$ if $\{\phi(\cdot-hk) : k \in \mathbb{Z}\}$ is a Riesz basis for $V_h(\phi)$, \textit{i.e.},
$\overline{span}\{\phi(\cdot-hk): k\in\mathbb{Z}\}=V_h(\phi)$ and there exist constants $A$, $B>0$ such that
\begin{equation}\label{rieszbasis}
A\sum_{k\in\mathbb{Z}}|d_k|^2\leq\big\|\sum_{k\in\mathbb{Z}}d_k\phi(\cdot-hk)\big\|^2\leq
B\sum_{k\in\mathbb{Z}}|d_k|^2,
\end{equation}
for all $(d_k)\in\ell^2(\mathbb{Z})$.    
It is well known that $\phi$ is a stable generator for $V_h(\phi)$ if and only if
\begin{equation}\label{eqn2.5}
 0<\|\Phi_{h}\|_0 \leq \|\Phi_{h}\|_{\infty} <\infty, 
\end{equation}
where $\|\Phi_{h}\|_0$ and $\|\Phi_{h}\|_\infty$ denote the essential infimum and supremum of the function
$\Phi_{h}(w):=\tfrac{1}{h}\sum_{n\in\mathbb{Z}}|\widehat{\phi}(w+\tfrac{n}{h})|^2$ in the interval $[0,1/h],$ respectively. 
The shift-invariant space $V_h(\phi)$ is called a reproducing kernel Hilbert space if for each $x\in\mathbb{R}$, there exists a unique element $K_x\in V_h(\phi)$ such that
$$f(x)=\langle f, K_x\rangle,~\text{for every}~f\in V_h(\phi).$$
The function $K(x,t):=K_x(t)=\langle K_x, K_t\rangle$ is called the reproducing kernel of $V_h(\phi)$. A set $\Lambda=\{x_{n}:n\in\mathbb{Z}\}$ of real numbers is said to be a set of stable sampling for $V_h(\phi)$ if there exist constants $A$, $B>0$ such that
\begin{eqnarray}\label{pap3eqn2.7}
A\| f\|^2\leq\ds\sum\limits_{n\in\mathbb{Z}}\ds|f(x_n)|^2 \leq B\| f\|^2, 
\end{eqnarray}
for all $f\in V_h(\phi)$. The numbers $A$ and $B$ are called sampling bounds. The following theorem which was implicitly proved by Janssen \cite{J95} and Ron and Shen \cite{Ron} establishes a fundamental link between Gabor analysis and the theory of sampling in shift-invariant spaces.
\begin{thm}\label{main}
Let $\phi$ be a stable generator for $V_{1/\beta}(\phi)$. Then
$\mathcal{G}(\phi,\alpha, \beta)$ forms a Gabor frame for $L^2(\mathbb{R})$ with frame bounds $A$ and $B$ if and only if $x+\alpha\mathbb{Z}$ is a set of stable sampling for $V_{1/\beta}(\phi)$ with sampling bounds ${\beta A}{\|\Phi_{1/\beta}\|^{-1}_{\infty}}$  and ${\beta B}{\|\Phi_{1/\beta} \|^{-1}_{0}}$ for almost all $x\in\mathbb{R}.$
\end{thm}

Gabor frames with compactly supported windows play a significant role in Gabor analysis due to their inherent time-frequency localization and they have received a lot of attention in recent years, see \cite{gwole, sign, mystery} and references therein. For a function $\phi$ with $\supp \phi\subseteq [0, L]$, $\mathcal{F}(\phi)$ is always a subset of $\left\{(\alpha,\beta)\in\mathbb{R}_{+}^2: \alpha\beta \le 1~\text{and}~\alpha \le L\right\}.$ The complete frame sets of the characteristic function of an interval \cite{abc, when} and the Haar function \cite{haar} are extremely complicated, in contrast to windows in Feichtinger algebra. In recent years, some progress in \cite{fscf, b3spline, ofsb, sign, olec, counter} has been made to  characterize the frame set for B-splines:
\begin{eqnarray}\label{Qmformula}
Q_1(x):=\chi_{\left[-\tfrac{1}{2},\tfrac{1}{2}\right]}(x)~\text{and}~Q_{m+1}(x):=(Q_m*Q_1)(x),~m\geq1.
\end{eqnarray}

By the Gaussian example, Daubechies initially conjectured in \cite{ID} that the frame set of a positive function with positive Fourier transform is $\{(\alpha, \beta)\in \mathbb{R}^2_+: \alpha\beta < 1\}$ but Janssen \cite{scewhf} disproved it. Gr\"{o}chenig et. al \cite{duke} achieved a breakthrough result and proved that the frame set is $\mathcal{F}(\phi) = \{(\alpha, \beta) \in \mathbb{R}^2_+: \alpha\beta < 1\}$ for any totally positive function $\phi$ of finite type $\ge 2$. Later, the authors in \cite{stsis} proved the same result for any totally positive functions of Gaussian type. Recall that a measurable function $\phi$ on $\mathbb{R}$ is totally positive if for every $n \in\mathbb{N}$ and every two sets of increasing numbers $x_1 < x_2 <\cdots< x_n$ and $y_1 < y_2 <\cdots< y_n$, the determinant of the matrix $[\phi(x_j-y_k)]_{j,k=1,\dots,n}$ is nonnegative. Schoenberg \cite{IJS} proved that the Fourier transform of an integrable totally positive function $\phi$ is of the form
\begin{eqnarray}\label{totallyft}
\widehat{\phi}(w)=ce^{-\eta w^2}e^{-2\pi i vw}\prod\limits_{j=1}^{N}\left(1+2\pi i v_jw\right)^{-1}e^{-2\pi iv_jw}
\end{eqnarray}
with $c>0$, $v, v_j\in\mathbb{R}$, $\eta\ge 0$, $N\in\mathbb{N}\cup\{\infty\}$, and $0<\eta+\sum_{j}v_j^2<\infty$.
A totally positive function $\phi$ is called of finite type if $\eta=0$ and $N\in\mathbb{N}$, of Gaussian type if $\eta\ne0$ and $N\in\mathbb{N}$, and of infinite type if $N=\infty$ in the factorization \eqref{totallyft}. Gr\"{o}chenig conjectured that if $\phi$ is a totally positive function other than the one-sided exponential, then the frame set is  $\mathcal{F}(\phi) = \{(\alpha, \beta) \in \mathbb{R}^2_+: \alpha\beta < 1\}$. 

Let $h_n(x)$ be the $n$-th Hermite function defined by
\begin{eqnarray}\label{hermitef}
h_n(x)=a_n e^{\pi x^2}\dfrac{d^n}{dx^n}e^{-2\pi x^2},~n=0,1,2,3,\dots,
\end{eqnarray}
where $a_n$ is chosen so that $\|h_n\|=1$. Since
$h_n$ belongs to Feichtinger algebra, the frame set $\mathcal{F}(h_n)$ is open in $\mathbb{R}_{+}^2$ and $\mathcal{F}(\phi)\subseteq \{(\alpha, \beta) \in\mathbb{R}^2_+: \alpha\beta< 1\}$. The authors in \cite{hermitegro, mathanna} proved that if $\alpha\beta <\tfrac{1}{n+1}$, then $\mathcal{G}(h_n, \alpha, \beta)$ is a frame. Finally, the authors \cite{Lyu} proved that the frame set of any odd window in Feichtinger algebra cannot contain the hyperbolas $\alpha\beta =\tfrac{p}{p+1}$ for any $p\in\mathbb{N}$. Based on these results, Gr\"{o}chenig \cite{mystery} conjectured that $\mathcal{F}(h_{2n})=\{(\alpha,\beta)\in\mathbb{R}_{+}^2: \alpha\beta < 1\}$ and $\mathcal{F}(h_{2n+1})\subset \{(\alpha,\beta)\in\mathbb{R}_{+}^2: \alpha\beta < 1,~\alpha\beta \ne\tfrac{p}{p+1},~p=1,2,\dots\}$. Later Lemvig \cite{hergabor} disproved it for $h_n$ with $n = 4m + 2$ and $n = 4m + 3$, $m\in\mathbb{N}_0$. 

\subsection{Our contribution.} The main contributions of this paper are listed as follows:
\begin{itemize}
\item[$(i)$] We provide a sufficient condition
for Gabor frames based on the connection between sampling theory in shift-invariant spaces and Gabor analysis (see Theorem \ref{frame}).
\item [$(ii)$] Using our sufficient condition, we establish a new region belonging to the frame set of B-splines, especially for $Q_2$ and $Q_3$. As a byproduct, we confirm that the region  $\left\{\alpha \in [ 2/9, 2/7 ],~\beta \in [4/(2+3\alpha), 2/(1+\alpha)],~\beta > 1\right\}$ is included in $\mathcal{F}(Q_2)$ which was suggested by numerical evidence \cite{ofsb}. Our result also covers the regions from \cite{ olec, counter} (see Figure \ref{bsplinesfig}).
\item [$(iii)$] We discuss the Gabor frame for a window $\phi\in L^1(\mathbb{R})$ such that its Fourier transform is of the form 
\begin{eqnarray}\label{type}
&\widehat{\phi}(w)=\psi(w)\prod\limits_{j=1}^{\infty}\left(1+2\pi i v_jw\right)^{-1}e^{-2\pi iv_jw},~v_j\in\mathbb{R},~0\le\sum_{j}v_j^2<\infty,
\end{eqnarray}
for some $\psi\in L^1(\mathbb{R})$. The window $\phi$ is said to be of \textit{type-I} if $\psi(w)=e^{-\pi w^2}$ and of \textit{type-II} if $\psi(w)=e^{-|w|}$ in \eqref{type}. Note that the functions of type-I are totally positive. We prove that if $\phi$ is of type-I or type-II and $\alpha\beta<1,$ then there exists a $\gamma>0$ depending on $\alpha\beta$ such that $\mathcal{G}(\phi_\gamma,\alpha,\beta)$ forms a frame for $L^2(\mathbb{R})$, where $\phi_\gamma(x)=\sqrt{\gamma}\phi(\gamma x)$ (see Theorems \ref{B1gammalem}, \ref{maintse}). This result is quite similar to \cite{approxgabor, approxgabor1}.
\item[$(iv)$] We obtain a new Gabor frame region for Hermite functions $h_n$ using MATLAB. Based on results from \cite{hergabor, Lyu}, we expect that our frame region for $h_{2n}$ might be close to the frame set (see Figure \ref{hermitefig}).  
\end{itemize}
\section{A sufficient condition for Gabor frames}
Let $\mathcal{A}$ denote the class of real-valued continuous functions $\phi$ satisfying the following conditions:
\begin{itemize}
\item [$(i)$] $\phi$ is differentiable except  at a finite  number of points.
\item [$(ii)$] For some $\epsilon>0$, $\phi^{(s)}(x)=\mathcal{O}(|x|^{-0.5-\epsilon})$ as $x\to\pm\infty$, $s=0,1$.
\item [$(iii)$] For all $h>0$, $\mathop{\rm{\esssup}}\limits_{w\in\left[0,1/h\right]}\ds\sum\limits_{l\in\mathbb{Z}}(w+\tfrac{l}{h})^{2}|\widehat{\phi}(w+\tfrac{l}{h})|^2<\infty$.
\end{itemize}
The Wiener space $W(\mathbb{R})$ is defined as
$$W(\mathbb{R}):=\left\{f ~\text{is continuous on}~\mathbb{R}:~\sum_{n\in\mathbb{Z}}\max\limits_{x\in[0,1]}|f(x+n)|<\infty\right\}.$$ 
If $\phi$ and $\widehat{\phi}$ belong to $W(\mathbb{R})$, then the Poisson summation formula \cite{time} 
\begin{eqnarray}\label{poissonsum}
\sum_{n\in\mathbb{Z}}\phi(x+\nu n)=\dfrac{1}{\nu}\sum_{n\in\mathbb{Z}}\widehat{\phi}\left(\dfrac{n}{\nu}\right)e^{2\pi i nx/\nu},~\nu>0,
\end{eqnarray}
holds for all $x\in\mathbb{R}$ with absolute convergence of both sums. If $\phi\in \mathcal{A}$ is a stable generator for $V_h(\phi)$, then we can show that $V_h(\phi)$ is a reproducing kernel Hilbert space (see, for example, \cite{Liu}). We define
\begin{equation}\label{Bphihs}
B_{\phi,h}(w):=\dfrac{\sum\limits_{l\in\mathbb{Z}}(w+\tfrac{l}{h})^{2}|\widehat{\phi}(w+\tfrac{l}{h})|^2}{\sum\limits_{l\in\mathbb{Z}}|\widehat{\phi}(w+\tfrac{l}{h})|^2}~\text{and}~M_{\phi,h}:=\mathop{\rm{\esssup}}\limits_{w\in\left[0,1/h\right]}B_{\phi,h}(w).
\end{equation}
Let $\phi_\gamma(x)=\sqrt{\gamma}\phi(\gamma x).$ It is easy to check that $M_{\phi_\gamma,h}=\gamma^2 M_{\phi,h\gamma}$ and 
\begin{eqnarray}\label{Msh}
M_{\phi,h}\geq\dfrac{\sum\limits_{l\in\mathbb{Z}}(\tfrac{1}{2h}+\tfrac{l}{h})^{2}|\widehat{\phi}(w+\tfrac{l}{h})|^2}{\sum\limits_{l\in\mathbb{Z}}|\widehat{\phi}(w+\tfrac{l}{h})|^2}\geq\dfrac{1}{4h^{2}}.
\end{eqnarray}
\begin{thm}\label{berns}
If $\phi\in \mathcal{A}$ is a stable generator for $V_{h}(\phi)$,
then
\begin{eqnarray}\label{bern}
\|f^{\prime}\|\leq 2\pi\sqrt{M_{\phi,h}}\|f\|,
\end{eqnarray}
for every $f\in V_{h}(\phi)$. Moreover, the constant $M_{\phi,h}$ depending on $\phi$ and $h$ is sharp. 
\end{thm}
The proof of the above theorem follows similar lines as in the proof of
Theorem 1 in \cite{bebenko} (see also \cite{perturb, Liu}).

Let $\{x_n:n\in\mathbb{Z}\}$, $\cdots<x_{n-1}<x_n<x_{n+1}<\cdots$, be a sampling set with $\lim\limits_{n\rightarrow \pm\infty}x_n=\pm \infty$. The sampling density is measured by the maximal gap between two consecutive samples, \textit{i.e.,} $\delta=\sup\limits_{n}(x_{n+1}-x_n).$ Consider the operator $P:L^2(\mathbb{R})\to V_h(\phi)$ by
\begin{eqnarray}\label{pap2eqn2.6}
(Pf)(x):=\langle f, K_x\rangle,
\end{eqnarray}
where $K_x(t)$ is the reproducing kernel of $V_h(\phi)$. Then $P$ is an orthogonal projection of $L^2(\mathbb{R})$ onto $V_h(\phi)$. Define the approximation operator $\mathrm{A}$ on $V_h(\phi)$ by 
$$\mathrm{A}f=P\left(\sum\limits_{n\in\mathbb{Z}}f(x_n)\chi_{[y_n,y_{n+1}]}\right),~\text{where}~y_n=\dfrac{x_n+x_{n+1}}{2}.$$
By using the same argument as in \cite{rais}, we can easily show the following estimate 
\begin{eqnarray}\label{faf}
\|f-Af\|^2
&\leq&\left\|\sum\limits_{n\in\mathbb{Z}}[f-f(x_n)]\chi_{[y_n,y_{n+1}]}\right\|^2\leq 4\delta^{2}M_{\phi,h}\|f\|^2,
\end{eqnarray}
from Theorem \ref{berns}. If $\delta<\tfrac{1}{2\sqrt{M_{\phi,h}}},$ then  $\|I-A\|_{op}<1$. Consequently, we obtain the following result from the Neumann theorem for the invertibility of an operator.
\begin{thm}\label{4.5}
Let $\phi\in \mathcal{A}$ be a stable generator for $V_h(\phi)$. If $\delta=\sup\limits_{n\in\mathbb{Z}}(x_{n+1}-x_n)<\tfrac{1}{2\sqrt{M_{\phi,h}}}$, then for every $f\in V_{h}(\phi)$
\begin{eqnarray}\label{af}
\|f-\mathrm{A}f\|\leq 2\delta\sqrt{M_{\phi,h}}\|f\|.
\end{eqnarray}
 Consequently, $\mathrm{A}$ is a bounded invertible operator on $V_{h}(\phi)$ with bounds
\begin{eqnarray}
\|\mathrm{A}f\|\leq\left(1+2\delta\sqrt{M_{\phi,h}}\right)\|f\|
\end{eqnarray}
and 
\begin{eqnarray}
\|\mathrm{A}^{-1}f\|\leq\left(1-2\delta\sqrt{M_{\phi,h}}\right)^{-1}\|f\|.
\end{eqnarray}
\end{thm} 
\begin{thm}\label{4.6}
Let $\phi\in \mathcal{A}$ be a stable generator for $V_h(\phi)$. If $\delta=\sup\limits_{n\in\mathbb{Z}}(x_{n+1}-x_n)<\tfrac{1}{2\sqrt{M_{\phi,h}}}$, then for every $f\in V_{h}(\phi)$ 
\begin{eqnarray}\label{bound}
\left(1-2\delta\sqrt{M_{\phi,h}}\right)^{2}\|f\|^2\leq\sum\limits_{n\in\mathbb{Z}}w_n|f(x_n)|^2 \leq \left(1+2\delta\sqrt{M_{\phi,h}}\right)^{2}\|f\|^2,
\end{eqnarray}
where $w_n=(x_{n+1}-x_{n-1})/2.$
\end{thm}
\begin{proof}
It follows from Theorem \ref{4.5} that
\begin{align}
\|f\|^2&=\|A^{-1}Af\|^2\nonumber\\
&\le\left(1-2\delta\sqrt{M_{\phi,h}}\right)^{-2}\left\|P\left(\sum_{n\in\mathbb{Z}}f(x_n)\chi_{[y_n,y_{n+1}]}\right)\right\|^2\nonumber\\
&\le\left(1-2\delta\sqrt{M_{\phi,h}}\right)^{-2}\left\|\sum_{n\in\mathbb{Z}}f(x_n)\chi_{[y_n,y_{n+1}]}\right\|^2\nonumber.
\end{align}
Since the characteristic functions $\chi_{[y_n,y_{n+1}]}$ have mutually disjoint support, we have
\begin{eqnarray*}
\left\|\sum\limits_{n\in\mathbb{Z}}f(x_n)\chi_{[y_n,y_{n+1}]}\right\|^2=\sum\limits_{n\in\mathbb{Z}}\int\limits_{y_{n}}^{y_{n+1}}|f(x_n)|^2~dx=\sum\limits_{n\in\mathbb{Z}}|f(x_n)|^2w_n
\end{eqnarray*}
and hence the left hand inequality of \eqref{bound} holds. Using the estimate \eqref{faf}, we get
\begin{align*}
\sum\limits_{n\in\mathbb{Z}}|f(x_n)|^2w_n&\le\left(\|f\|+\left\|\sum\limits_{n\in\mathbb{Z}}[f-f(x_n)]\chi_{[y_n,y_{n+1}]}\right\|\right)^2\\
&\le\left(1+2\delta\sqrt{M_{\phi,h}}\right)^2\|f\|^2.
\end{align*}
\end{proof}
If we choose $x_n=x+\alpha n$ and $h=1/\beta$ in Theorem \ref{4.6}, then $x+\alpha\mathbb{Z}$ is a set of stable sampling for $V_{1/\beta}(\phi)$ with sampling bounds $\tfrac{1}{\alpha} \left(1-2\alpha\sqrt{M_{\phi,1/\beta}}\right)^2$  and $\tfrac{1}{\alpha} \left(1+2\alpha\sqrt{M_{\phi,1/\beta}}\right)^2$ for almost all $x\in\mathbb{R}$. Consequently, we obtain the following result from Theorem \ref{main}.
\begin{thm}\label{frame}
Let $\phi\in \mathcal{A}$ be a stable generator for $V_{1/\beta}(\phi)$. If $0<\alpha<\tfrac{1}{2\sqrt{M_{\phi,1/\beta}}},$ then $\mathcal{G}(\phi,\alpha,\beta)$ forms a frame for $L^2(\mathbb{R})$ with frame bounds 
\begin{eqnarray*}
A(\alpha,\beta)=\tfrac{1}{\alpha}\left(1-2\alpha\sqrt{M_{\phi,1/\beta}}\right)^2\esssup\limits_{\xi\in[0,\beta]}\sum\limits_{k\in\mathbb{Z}}|\widehat{\phi}(\xi+\beta k)|^2,
\end{eqnarray*}
and
\begin{eqnarray*}
B(\alpha,\beta)=\tfrac{1}{\alpha}\left(1+2\alpha\sqrt{M_{\phi,1/\beta}}\right)^2\essinf\limits_{\xi\in[0,\beta]}\sum\limits_{k\in\mathbb{Z}}|\widehat{\phi}(\xi+\beta k)|^2.
\end{eqnarray*}
\end{thm}
\subsection{Painless Non-Orthogonal Expansions}
Let $\phi$ be continuous and differentiable except at a finite number of points with $\supp{\phi}=[-\sigma,\sigma]$. If $|{\phi}(w)|>0$ on the interior of $[-\sigma,\sigma]$, then $\widehat{\phi}$ is a stable generator for $V_{1/\alpha}(\widehat{\phi})$ whenever $0<\alpha<2\sigma$. If $f(x)=\sum\limits_{k\in\mathbb{Z}}c_k\widehat{\phi}(x-k/\alpha)$, then it follows from Plancherel’s identity that
\begin{align*}
\|f^{\prime}\|^2=(2\pi)^{2}\int_{-\sigma}^{\sigma}\left|\sum\limits_{k\in\mathbb{Z}}c_k e^{-2\pi ikw/\alpha}w{\phi}(w)\right|^2 dw\le(2\pi\sigma)^{2}\|f\|^2.
\end{align*}
Therefore $M_{\widehat{\phi},1/\alpha}\leq\sigma^2$. We know that $\mathcal{G}(\phi,\alpha, \beta)$ is a frame for $L^2(\mathbb{R})$ with frame bounds $A$ and $B$ if and only if $\mathcal{G}(\widehat{\phi},\beta,\alpha)$ is a frame for $L^2(\mathbb{R})$ with same frame bounds. Consequently, we obtain the following result \cite{painless} as an immediate application of Theorem \ref{frame}.
\begin{cor}
Let $\phi$ be continuous, differentiable except at a finite number of points with support in $[-\sigma,\sigma]$, and $|{\phi}(w)|>0$ on the interior of $[-\sigma,\sigma].$ If $0<\alpha<2\sigma$ and $0<\beta<\tfrac{1}{2\sigma},$ then $\mathcal{G}(\phi, \alpha, \beta)$ is a frame for $L^2(\mathbb{R}).$
\end{cor}
\section{Gabor frames with B-splines} 
The B-spline $Q_m$ defined in \eqref{Qmformula} belongs to $C_c^{m-2}(\mathbb{R})$ with support  $\left[-\tfrac{m}{2},\tfrac{m}{2}\right]$ and 
its Fourier transform is given by $\widehat{Q_m}(w)=\left(\tfrac{\sin \pi w}{\pi w}\right)^m.$ The B-splines and their derivatives can be computed by the following formulae
$$Q_m(x)=\dfrac{1}{(m-1)!}\displaystyle\sum_{j=0}^{m}(-1)^j
\binom{m}{j} \left(x+\tfrac{m}{2}-j\right)_+^{m-1},~x_+=\max(0,x),~m\geq 2,$$
and 
$$\hspace{-1.55cm} Q_m^\prime(x)=Q_{m-1}(x+\tfrac{1}{2})-Q_{m-1}(x-\tfrac{1}{2}),~m >2,$$
respectively. Dai and Sun \cite{abc} provided the complete frame set of $Q_1$. In this section, we mainly focus on the B-splines of order $\ge 2$. When $\beta\in\mathbb{R}\setminus\{2,3,4,\dots\},$  $Q_m$ is a stable  generator for $V_{1/\beta}(Q_m).$ 
To find a Gabor frame region for $Q_2$, we write explicitly 
\begin{align*}
Q_2(x)=
\begin{cases}
1+x~&\text{if}~x\in[-1,0],\\
1-x~&\text{if}~x\in[0, 1],\\
0~&\text{otherwise},
\end{cases}~\text{and}~
Q_4(x)=
\begin{cases}
    \dfrac{(x+2)^3}{6}~&\text{if}~x\in[-2,-1],\\
    \dfrac{-3x^3-6x^2+4}{6}~&\text{if}~x\in[-1,0],\\
    \dfrac{3x^3-6x^2+4}{6}~&\text{if}~x\in[0,1],\\
    \dfrac{(2-x)^3}{6}~&\text{if}~x\in[1,2],\\
    0~&\text{otherwise}.
\end{cases}
\end{align*}
Throughout this section, we denote $y=2\pi w/\beta$ and $u=\cos{y}.$ 
Using the Poisson summation formula, we get 
\begin{align*}
B_{Q_m,1/\beta}(w)&=-\dfrac{1}{4\pi^2}\dfrac{\sum\limits_{n=-\lceil m\beta \rceil}^{\lceil m\beta \rceil}Q_{2m}^{\prime\prime}(\tfrac{n}{\beta})e^{2\pi in w/\beta}}{\sum\limits_{n=-\lceil m\beta \rceil}^{\lceil m\beta \rceil}Q_{2m}(\tfrac{n}{\beta})e^{2\pi in w/\beta}}\\
&=\dfrac{1}{4\pi^2}\dfrac{\sum\limits_{n=-\lceil m\beta \rceil}^{\lceil m\beta \rceil}{[2Q_{2m-2}(\tfrac{n}{\beta})-Q_{2m-2}(1+\tfrac{n}{\beta})-Q_{2m-2}(\tfrac{n}{\beta}-1)]}e^{2\pi in w/\beta}}{Q_{2m}(0)+2\sum\limits_{n=1}^{\lceil m\beta \rceil}Q_{2m}(\tfrac{n}{\beta})\cos{ny}}\\
&=\dfrac{1}{2\pi^2}\dfrac{c+\sum\limits_{n=1}^{\lceil m\beta \rceil}{[2Q_{2m-2}(\tfrac{n}{\beta})-Q_{2m-2}(1+\tfrac{n}{\beta})-Q_{2m-2}(\tfrac{n}{\beta}-1)]}\cos{n y}}{Q_{2m}(0)+2\sum\limits_{n=1}^{\lceil m\beta \rceil}Q_{2m}(\tfrac{n}{\beta})\cos{n y}},
\end{align*}
where $c=Q_{2m-2}(0)-Q_{2m-2}(1)$. When $m=2$ and $0<\beta<2$, we have 
\begin{align}\label{Q2bw}
B_{Q_2,1/\beta}(w)=\dfrac{1}{4\pi^2}\dfrac{N_\beta(w)}{D_\beta(w)},
\end{align}
where
\begin{align*}
N_\beta(w)&=1+\sum\limits_{n=1}^3\left[2Q_2\left(\dfrac{n}{\beta}\right)-Q_2\left(1-\dfrac{n}{\beta}\right)\right]\cos{ny}
\end{align*}
and
\begin{eqnarray*}
D_\beta(w)=\dfrac{1}{3}+\sum\limits_{n=1}^3Q_4\left(\dfrac{n}{\beta}\right)\cos{n y}.
\end{eqnarray*}
To find the exact value of $M_{Q_2,1/\beta}$ for $0<\beta<2,$ we need some well-known results from the theory of distribution of the zeros of polynomials. Let $V(a_0, a_1,\dots,a_n)$ denote the number of sign changes in the reduced sequence obtained by ignoring all vanishing elements of a finite sequence of real numbers $a_0, a_1,\dots,a_n$.  
For a real polynomial $f$ of degree $n$ and an interval $I$, let $\mathcal{N}_fI$ denote the number of zeros of $f$ (counting multiplicity) in $I$. Furthermore,
$$V_f(x):=V(f(x),f^\prime(x),\dots, f^{(n)}(x)).$$
\begin{thm}[\textbf{Budan-Fourier}]\cite{atop}\label{FB}
If $f$ is a real polynomial of degree $n$, then for any interval $(a,b]$,
$$\mathcal{N}_f(a,b]=V_f(a)-V_f(b)-2k$$
for some $k\in\mathbb{N}_0.$
\end{thm}
\begin{algo}\cite{atop}
Let $f$ be a real polynomial of degree $n$. Define the sequence of polynomials $f_0, f_1,\dots,f_n$ such that
\begin{align*}
f_0(x)=&f(x),\\
f_1(x)=&f^\prime(x),\\
f_{j+1}(x)=&-\mathrm{rem}(f_{j-1}(x),f_j(x)),~j=1,\dots,n-1.
\end{align*}
If $f(a)f(b)\ne 0$, then 
$$\mathcal{N}_f[a,b]=V(f_0(a),f_1(a),\dots, f_n(a))-V(f_0(b),f_1(b),\dots, f_n(b)).$$
\end{algo}
\begin{thm}
For each fixed $0<\beta<2,$ $B_{Q_2,1/\beta}(w)$ attains its maximum at $w=\tfrac{\beta}{2}$ in $[0,\beta]$. Consequently, if $0<\alpha<\tfrac{1}{2\sqrt{M_{Q_2,1/\beta}}}$, then $\mathcal{G}(Q_2,\alpha,\beta)$ forms a frame for $L^2(\mathbb{R}),$ where
\begin{align}\label{MQ2less2}
M_{Q_2,1/\beta}=\dfrac{1}{4\pi^2}\dfrac{1-2Q_2\left(\tfrac{1}{\beta}\right)+Q_2\left(1-\tfrac{1}{\beta}\right)-Q_2\left(1-\tfrac{2}{\beta}\right)+Q_2\left(1-\tfrac{3}{\beta}\right)}{\tfrac{1}{3}-Q_4\left(\tfrac{1}{\beta}\right)+Q_4\left(\tfrac{2}{\beta}\right)-Q_4\left(\tfrac{3}{\beta}\right)}.
\end{align}
\end{thm}
\begin{proof}
To prove our result, we show that $B_{Q_2,1/\beta}(w)$ is an increasing function in $[0,\beta/2].$ Since 
$$B_{Q_2,1/\beta}^\prime(w)=\dfrac{1}{4\pi^2}\dfrac{(N_\beta^\prime D_\beta -N_\beta D_\beta^\prime)(w)}{D_{\beta}^{2}(w)},$$ it is enough to show that $(N_\beta^\prime D_\beta -N_\beta D_\beta^\prime)(w)\ge 0$ in $[0,\beta/2].$ We use the following identities in our proof.
\begin{subequations}\label{identity}
 \begin{align}
 \sin{2y}&=2u\sin{y},&\sin{3y}&=(4u^2-1)\sin{y},\label{eq11}\\
\sin{4y}&=(8u^3-4u)\sin{y},&
\sin{5y}&=(16u^4-12u^2+1)\sin{y}.\label{eq12}
 \end{align}
\end{subequations} 
When $0<\beta\le 1/2$, it is clear that $N_\beta(w)=1$ and $D_\beta(w)=\tfrac{1}{3}$. Thus $M_{Q_2,1/\beta}={3}/{4\pi^2}.$
For $1/2<\beta<2$, we divide the proof into three cases.
\newline
\underline{\textbf{Case: 1}} When $1/2<\beta\le 1,$ we have
\begin{align*}
N_\beta(w)=1-Q_2\left(1-\dfrac{1}{\beta}\right)\cos{y}
=1-\left(2-\dfrac{1}{\beta}\right)\cos{y}
\end{align*}
and
\begin{align*}
D_\beta(w)=\dfrac{1}{3}+Q_4\left(\dfrac{1}{\beta}\right)\cos{y}=\dfrac{1}{3}+\dfrac{(2\beta-1)^3}{6\beta^3}\cos{y}.
\end{align*}
In this case, $(N_\beta^\prime D_\beta-N_\beta D_\beta^\prime)(w)$
\begin{align*}
&=\dfrac{2\pi}{\beta}\left[\left(2-\dfrac{1}{\beta}\right)\sin{y}\left[\dfrac{1}{3}+\dfrac{(2\beta-1)^3}{6\beta^3}\cos{y}\right]+\left[1-\left(2-\dfrac{1}{\beta}\right)\cos{y}\right]\dfrac{(2\beta-1)^3}{6\beta^3}\sin{y}\right]\\
&=\dfrac{2\pi}{\beta}\left[\dfrac{2\beta-1}{3\beta}+\dfrac{(2\beta-1)^3}{6\beta^3}\right]\sin{y}\\
&=\dfrac{\pi(2\beta-1)(6\beta^2-4\beta+1)}{3\beta^4}\sin{y}\ge 0
\end{align*}
in $[0,\beta/2]$ because the quadratic polynomial $6\beta^2-4\beta+1$ is positive.
\newline
\underline{\textbf{Case: 2}} When $1<\beta\le 3/2,$ we have
\begin{align*}
N_\beta(w)&=1+\left[2Q_2\left(\dfrac{1}{\beta}\right)-Q_2\left(1-\dfrac{1}{\beta}\right)\right]\cos{y}-Q_2\left(1-\dfrac{2}{\beta}\right)\cos{2y}\\
&=1+\left(2-\dfrac{3}{\beta}\right)\cos{y}-\left(2-\dfrac{2}{\beta}\right)\cos{2y}\\
&=\dfrac{1}{\beta}\left[\beta+(2\beta-3)\cos{y}-2(\beta-1)\cos{2y}\right]
\end{align*}
and
\begin{align*}
\hspace{-1cm} D_\beta(w)&=\dfrac{1}{3}+Q_4\left(\dfrac{1}{\beta}\right)\cos{y}+Q_4\left(\dfrac{2}{\beta}\right)\cos{2y}\\
&=\dfrac{1}{3}+\dfrac{4\beta^3-6\beta+3}{6\beta^3}\cos{y}+\dfrac{4}{3}\left(\dfrac{
\beta-1}{\beta}\right)^3\cos{2y}\\
&=\dfrac{1}{6\beta^3}[2\beta^3+(4\beta^3-6\beta+3)\cos{y}+8(\beta-1)^3\cos{2y}].
\end{align*}
Consequently, $(N_\beta^\prime D_\beta)(w)$
\begin{align*}
&=\dfrac{\pi}{3\beta^5}\left\{-(2\beta-3)\sin{y}+4(\beta-1)\sin{2y}\right\}\times\\&\hspace{4cm}\left\{2\beta^3+(4\beta^3-6\beta+3)\cos{y}+8(\beta-1)^3\cos{2y}\right\}\\
&=\dfrac{\pi}{3\beta^5}\big[-(4\beta^4-6\beta^3)\sin{y}+8(\beta^4-\beta^3)\sin{2y}+32(\beta-1)^4\sin{2y}\cos{2y}\\&\hspace{5cm}-(8\beta^4-12\beta^3-12\beta^2+24\beta-9)\sin{y}\cos{y}\\&\hspace{5.2cm}+4(4\beta^4-4\beta^3-6\beta^2+9\beta-3)\sin{2y}\cos{y}\\&\hspace{4.7cm}-8(2\beta^4-9\beta^3+15\beta^2-11\beta+3)\sin{y}\cos{2y}\big]
\end{align*}
and $(N_\beta D_\beta^\prime)(w)$
\begin{align*}
&=\dfrac{\pi}{3\beta^5}\left[\beta+(2\beta-3)\cos{y}-2(\beta-1)\cos{2y}\right]\times\\&\hspace{5cm}\left[-(4\beta^3-6\beta+3)\sin{y}-16(\beta-1)^3\sin{2 y}\right]\\
&=\dfrac{\pi}{3\beta^5}\big[-(4\beta^4-6\beta^2+3\beta)\sin{y}-16\beta(\beta-1)^3\sin{2 y}\\&\hspace{6 cm}-(8\beta^4-12\beta^3-12\beta^2+24\beta-9)\sin{y}\cos{y}\\&\hspace{5.5 cm}-16(2\beta^4-9\beta^3+15\beta^2-11\beta+3)\sin{2y}\cos{y}\\&\hspace{1.3 cm}+(8\beta^4-8\beta^3-12\beta^2+18\beta-6)\sin{y}\cos{2y}+32(\beta-1)^4\sin{2y}\cos{2y}\big].
\end{align*}
Applying the identities \eqref{eq11}, we obtain
$\dfrac{3\beta^5}{\pi}(N_\beta^\prime D_\beta-N_\beta D_\beta^\prime)(w)$
\begin{align}\label{case3}
&=(6\beta^3-6\beta^2+3\beta)\sin{y}+(24\beta^4-56\beta^3+48\beta^2-16\beta)\sin{2 y}\nonumber\\&\hspace{4.55cm}+(-24\beta^4+80\beta^3-108\beta^2+70\beta-18)\sin{y}\cos{2 y}\nonumber\\&\hspace{4.5cm}+(48\beta^4-160\beta^3+216\beta^2-140\beta+36)\sin{2y}\cos{y}\nonumber\\
&=(6\beta^3-6\beta^2+3\beta)\sin{y}+(24\beta^4-56\beta^3+48\beta^2-16\beta)\sin{2 y}\nonumber\\&\hspace{4.7cm}+(-12\beta^4+40\beta^3-54\beta^2+35\beta-9)(\sin{3 y}-\sin{y})\nonumber\\&\hspace{4.5cm}+(24\beta^4-80\beta^3+108\beta^2-70\beta+18)(\sin{3 y}+\sin{y})\nonumber\\
&=a_1\sin{y}+a_2\sin{2y}+a_3\sin{3 y}\nonumber\\
&=(a_1+2a_2u+a_3(4u^2-1))\sin{y},
\end{align}
where $ a_1,  a_2$, and $a_3$ are polynomials in $\beta$ given by 
\begin{align*}
a_1&=36\beta^4-114\beta^3+156\beta^2-102\beta+27,\\
a_2&=24\beta^4-56\beta^3+48\beta^2-16\beta,\\
a_3&=12\beta^4-40\beta^3+54\beta^2-35\beta+9.
\end{align*}
Let $P_\beta(u)=4a_3u^2+2a_2u+a_1-a_3.$
Then \eqref{case3} becomes
\begin{align}\label{case3f}
 \dfrac{3\beta^5}{\pi}(N_\beta^\prime D_\beta-N_\beta D_\beta^\prime)(w)=P_\beta(u)\sin{y}.   
\end{align}
Differentiating $P_\beta(u)$ two times with respect to $u$, we get $$P_\beta^\prime(u)=8a_3u+2a_2~\text{and}~P_\beta^{\prime\prime}(u)=8a_3.$$ 
Consider
\begin{align*}
P_\beta(1)&=3a_3+2a_2+a_1=120\beta^4-346\beta^3+414\beta^2-239\beta+54,  \\
P_\beta^\prime(1)&=8a_3+2a_2
=24(\beta-1)(6\beta^3-12\beta^2+10\beta-3),\\
P_\beta^{\prime}(1)&=P_\beta^{\prime\prime}(-1)=8a_3=8(\beta-1)(12\beta^3-28\beta^2+26\beta-9),\\
P_\beta(-1)&=3a_3-2a_2+a_1=(2-\beta)(-24\beta^3+74\beta^2-74\beta+27),\\
P_\beta^{\prime}(-1)&=-8a_3+2a_2=8(\beta-1)(-6\beta^3+20\beta^2-22\beta+9).
\end{align*}
For all $\beta\in(1,3/2]$, the sign of the sequence $\{P_\beta, P_\beta^\prime, P_\beta^{\prime\prime}\}$ is determined at the end points of the interval $[-1, 1]$ using Budan-Fourier Theorem. Please see Table \ref{case2} in the Appendix. We observe from Table \ref{case2} that 
\begin{align*}
\mathcal{N}_{P_\beta}[-1, 1]\le V_{P_\beta}(-1)-V_{P_\beta}(1)=0,
\end{align*}
using Budan-Fourier Theorem. Since $P_\beta(1)>0$
and $\sin{y}\ge 0$ in $[0,\pi]$, we conclude from \eqref{case3f} that $(N_\beta^\prime D_\beta-N_\beta D_\beta^\prime)(w)$ is nonnegative in $[0,\beta/2]$.
\newline
\underline{\textbf{Case: 3}} When $3/2< \beta<2$, we have 
$N_\beta(w)$
\begin{align*}
&=1+\left[2Q_2\left(\dfrac{1}{\beta}\right)-Q_2\left(1-\dfrac{1}{\beta}\right)\right]\cos{y}-Q_2\left(1-\dfrac{2}{\beta}\right)\cos{2y}-Q_2\left(1-\dfrac{3}{\beta}\right)\cos{3y}
\end{align*}
\begin{align*}
&=1+\left(2-\dfrac{3}{\beta}\right)\cos{y}-\left(2-\dfrac{2}{\beta}\right)\cos{2y}-\left(2-\dfrac{3}{\beta}\right)\cos{3y}\\
&=\dfrac{1}{\beta}\left[\beta+(2\beta-3)\cos{y}-(2\beta-2)\cos{2y}-(2\beta-3)\cos{3y}\right]
\end{align*}
and 
\begin{align*}
D_\beta(w)&=\dfrac{1}{3}+Q_4\left(\dfrac{1}{\beta}\right)\cos{y}+Q_4\left(\dfrac{2}{\beta}\right)\cos{2y}+Q_4\left(\dfrac{3}{\beta}\right)\cos{3y}\\
&=\dfrac{1}{3}+\dfrac{4\beta^3-6\beta+3} 
 {6\beta^3}\cos{y}+\dfrac{4}{3}\left(\dfrac{
\beta-1}{\beta}\right)^3\cos{2y}+\dfrac{
(2\beta-3)^3}{6\beta^3}\cos{3y}\\
&=\dfrac{1}{6\beta^3}\left[2\beta^3+(4\beta^3-6\beta+3)\cos{y}+8(\beta-1)^3\cos{2y}+(2\beta-3)^3\cos{3y}\right].
\end{align*}
Let us denote $A=4\beta^3-6\beta+3$, $B=8(\beta-1)^3$, $ C=(2\beta-3)^3$, $M=2\beta-3$, and $R=2-2\beta$. Then
\begin{eqnarray*}
(N_\beta^\prime D_\beta)(w)=\dfrac{\pi}{3\beta^5}(-M\sin{y}-2R\sin{2y}+3M\sin{3y})\\&\hspace{-2cm}\times(2\beta^3+A\cos{y}+B\cos{2y}+C\cos{3y})
\end{eqnarray*}
and
\begin{align*}
(N_\beta D_\beta^\prime)(w)=\dfrac{\pi}{3\beta^5}(\beta+M\cos{y}+R\cos{2y}-M\cos{3y})\\&\hspace{-1.6cm}\times(-A\sin{y}-2B\sin{2y}-3C\sin{3y})
\end{align*}
Applying the identities \eqref{identity}, we obtain
$\dfrac{3\beta^5}{\pi}(N_\beta^\prime D_\beta-N_\beta D_\beta^\prime)(w)$
\begin{align}\label{case4}
&=(A\beta-2\beta^3 M)\sin{y}+(2\beta B-4\beta^3 R)\sin{2y}+(3\beta C+6\beta^3 M)\sin{3y}\nonumber\\&\hspace{0.9cm}+(AR-BM)\sin{y}\cos{2y}-(A+C)M\sin{y}\cos{3y}+2(BM-AR)\sin{2y}\cos{y}\nonumber\\&-2(BM+CR)\sin{2y}\cos{3y}+3(A+C)M\sin{3y}\cos{y}+3(BM+CR)\sin{3y}\cos{2y}\nonumber\\
&=(A\beta-2\beta^3 M)\sin{y}+(2\beta B-4\beta^3 R)\sin{2y}+(3\beta C+6\beta^3 M)\sin{3y}\nonumber\\&+(BM-AR)(2\sin{2y}\cos{y}-\sin{y}\cos{2y})+(A+C)M(3\sin{3y}\cos{y}-\sin{y}\cos{3y})\nonumber\\&\hspace{6.6cm}+(BM+CR)(3\sin{3y}\cos{2y}-2\sin{2y}\cos{3y})\nonumber\\
&=(A\beta-2\beta^3 M)\sin{y}+(2\beta B-4\beta^3 R)\sin{2y}+(3\beta C+6\beta^3 M)\sin{3y}\nonumber\\&\hspace{3.5cm}+(BM-AR)(\tfrac{1}{2}\sin{3y}+\tfrac{3}{2}\sin{y})+(A+C)M(\sin{4y}+2\sin{2y})\nonumber\\&\hspace{8.9cm}+(BM+CR)(\tfrac{1}{2}\sin{5y}+\tfrac{5}{2}\sin{y})\nonumber\\
&=b_1\sin{y}+b_2\sin{2y}+b_3\sin{3y}+b_4\sin{4y}+b_5\sin{5y}\nonumber\\
&=[b_1+2b_2u+b_3(4u^2-1)+b_4(8u^3-4u)+b_5(16u^4-12u^2+1)]\sin{y},
\end{align}
where $b_1,  b_2, b_3, b_4$, and $b_4$ are polynomials in $\beta$ given by 
\begin{align*}
b_1&=A\beta-2\beta^3 M+\tfrac{3}{2}(BM-AR)+\tfrac{5}{2}(BM+CR)\\
&=36\beta^4-74\beta^3+6\beta^2+83\beta-48,
\end{align*}
\begin{align*}
b_2&=2\beta B-4\beta^3 R+2(A+C)M=72\beta^4-272\beta^3+456\beta^2-400\beta+144,\\
b_3&=3\beta C+6\beta^3 M+\tfrac{1}{2}(BM-AR)=48\beta^4-166\beta^3+216\beta^2-116\beta+9,\\
b_4&=(A+C)M=24\beta^4-108\beta^3+204\beta^2-192\beta+72,\\
b_5&=\tfrac{1}{2}(BM+CR)=8\beta^3-30\beta^2+37\beta-15.
\end{align*}
Let $K_\beta(u):=16b_5u^4+8b_4u^3+(4b_3-12b_5)u^2+(2b_2-4b_4)u+b_1-b_3+b_5.$ Then \eqref{case4} becomes
\begin{align}\label{case4f}
\dfrac{3\beta^5}{\pi}(N_\beta^\prime D_\beta-N_\beta D_\beta^\prime)(w)=K_\beta(u)\sin{y}.
\end{align}
Differentiating $K_\beta(u)$ four times with respect to $u$, we get 
\begin{align*}
K_\beta^\prime(u)&=64b_5u^3+24b_4u^2+2(4b_3-12b_5)u+2b_2-4b_4,&K_\beta^{\prime\prime\prime}(u)&=384b_5u+48b_4,\\
K_\beta^{\prime\prime}(u)&=192b_5u^2+48b_4u+2(4b_3-12b_5),&
K_\beta^{(4)}(u)&=384b_5.
\end{align*}
Consider
\begin{align*}
K_\beta(1)&=5b_5+4b_4+3b_3+2b_2+b_1=4(105\beta^4-377\beta^3+558\beta^2-412\beta+120),\\
K_\beta^\prime(1)&=40b_5+20b_4+8b_3+2b_2
=8(126\beta^4-464\beta^3+690\beta^2-511\beta+150),\\
K_\beta^{\prime\prime}(1)&=168b_5+48b_4+8b_3=8(192\beta^4-646\beta^3+810\beta^2-491\beta+126),\\
K_\beta^{\prime\prime\prime}(1)&=384b_5+48b_4=192(\beta-1)(2\beta-3)(3\beta^2+2\beta-4),\\
K_\beta^{(4)}(1)&=K_\beta^{(4)}(-1)=384b_5=384(\beta-1)(2\beta-3)(4\beta-5)\\
K_\beta(-1)&=5b_5-4b_4+3b_3-2b_2+b_1=12(2-\beta)^3(5\beta-7),\\
K_\beta^\prime(-1)&=-40b_5+20b_4-8b_3+2b_2=8(2-\beta)(-30\beta^3+152\beta^2-254\beta+141),\\
K_\beta^{\prime\prime}(-1)&=168b_5-48b_4+8b_3=8(2-\beta)(96\beta^3-458\beta^2+722\beta-369),\\
K_\beta^{\prime\prime\prime}(-1)&=-384b_5+48b_4=192(\beta-1)(\beta-2)(2\beta-3)(3\beta-8).
\end{align*}
We now determine the sign of the sequence $\{K_\beta, K_\beta^\prime, K_\beta^{\prime\prime}, K_\beta^{\prime\prime\prime}\}$ at the end points of the interval $[-1, 1]$. It is easy to check that $K_\beta^{(4)}(1)$, $K_\beta(-1)$ and $K_\beta^{\prime\prime\prime}(-1)$ are nonnegative in $[3/2,2].$ For other elements except $K_\beta^{(4)}(1)$, $K_\beta(-1)$ and $K_\beta^{\prime\prime\prime}(-1)$, please see Table \ref{bigtable} in the Appendix. To determine the sign of $K^{\prime}_\beta(-1)$ and $K^{\prime\prime}_\beta(-1)$ in $(3/2,2]$, we use Sturm's algorithm. Define 
\begin{align*}
\psi_0(\beta)&=-30\beta^3+152\beta^2-254\beta+141,&f_0(\beta)&=96\beta^3-458\beta^2 + 722\beta-369,\\
\psi_1(\beta)&=-90\beta^2+304\beta-254,&
f_1(\beta)&=288\beta^2-916\beta + 722,\\
\psi_2(\beta)&=-\mathrm{rem}(\psi_0(\beta),\psi_1(\beta))&f_2(\beta)&=-\mathrm{rem}(f_0(\beta),f_1(\beta))\\
&=-\dfrac{244}{135}\beta+\dfrac{269}{135},&&=\dfrac{457}{108}\beta-\dfrac{2965}{216},\\
\psi_3(\beta)&=-\mathrm{rem}(\psi_1(\beta),\psi_2(\beta)),&f_3(\beta)&=-\mathrm{rem}(f_1(\beta),f_2(\beta)),\\
&=\dfrac{840645}{29768},&&=-\dfrac{163164888}{208849}.
\end{align*}
Then 
\begin{align*}
\mathcal{N}_{\psi_0}[1.5,2]&=V(\psi_0(1.5),\psi_1(1.5),\psi_2(1.5),\psi_3(1.5))-V(\psi_0(2),\psi_1(2),\psi_2(2),\psi_3(2))\\
&=V\left(\tfrac{3}{4}, -\tfrac{1}{2}, -\tfrac{97}{135}, \tfrac{840645}{29768})-V(1, -6, -\tfrac{73}{45}, \tfrac{840645}{29768}\right)\\
&=2-2=0
\end{align*}
and 
\begin{align*}
\mathcal{N}_{f_0}[1.5,2]&=V(f_0(1.5),f_1(1.5),f_2(1.5),f_3(1.5))-V(f_0(2),f_1(2),f_2(2),f_3(2))\\
&=V\left(\tfrac{15}{2},-4,-\tfrac{797}{108},-\tfrac{163164888}{208849})-V(11,42,-\tfrac{379}{72},-\tfrac{163164888}{208849}\right)\\
&=1-1=0.
\end{align*}
Since $\psi_0(1.5)=0.75>0$ and $f_0(1.5)=7.5>0$, $\psi_0$ and $f_0$ are positive in  $[3/2,2]$ which implies that $K^{\prime}_\beta(-1)$ and $K^{\prime\prime}_\beta(-1)$ are nonnegative in $(3/2,2]$. Hence 
$$\mathcal{N}_{K_\beta}[-1, 1]\le V_{K_\beta}(-1)-V_{K_\beta}(1)=0,$$
using Budan-Fourier Theorem.
Since $K_\beta(-1)>0$ and $\sin{y}\ge 0$ in $[-1,1]$, we conclude from \eqref{case4f} that $(N_\beta^\prime D_\beta-N_\beta D_\beta^\prime)(w)$ is nonnegative.
\end{proof}
We can also prove an analog of Theorem \ref{MQ2less2} for $Q_3$ using the Budan-Fourier theorem and Sturm's algorithm. We leave it as an
exercise for an interested reader. The MATLAB code for the Budan-Fourier theorem and Sturm's algorithm is listed in the Appendix. We verified that $B_{Q_2,1/\beta}(w)$ and $B_{Q_3,1/\beta}(w)$ do not necessarily attain their maximum at $w=\beta/2$ for all $\beta>2.$  Based on Theorem \ref{frame}, we find a new region belonging to the frame set for $Q_2$ and $Q_3$ using MATLAB. Please refer to Figure \ref{bsplinesfig}. 

Figure \ref{bsplinesfig} confirms that the region $\left\{\alpha \in \left[ \tfrac{2}{9},\tfrac{2}{7}\right],~\beta \in \left[\tfrac{4}{2+3\alpha}, \tfrac{2}{1+\alpha}\right],~\beta > 1\right\}$ is included in $\mathcal{F}(Q_2)$ which was suggested by numerical evidence \cite{ofsb}. Our results also cover the regions from \cite{olec, counter}. The authors in \cite{approxgabor} proved that for any $\alpha, \beta >0$ such that $\alpha\beta<1$, there exists a positive integer $m(\alpha, \beta)$ such that the function
$$g_m (x) :=\sqrt{\dfrac{m}{12}}Q_m\left(\sqrt{\dfrac{m}{12}}x\right)$$
generates a Gabor frame $\mathcal{G}(g_m,\alpha,\beta)$ whenever $m \ge m(\alpha,\beta)$. In contrast to the approach in \cite{approxgabor}, we will choose $Q_m(\gamma \cdot)$ instead of $g_m$ but the order of the B-spline is fixed. For $m\ge 2,$ the authors in \cite{sign} proved that if $\alpha,\beta> 0$ such that $\alpha\beta < 1,$ then $\mathcal{G}(Q_{m},\alpha,\beta)$ forms a Gabor frame for $L^2(\mathbb{R})$ if there exists a $k\in\mathbb{N}$ such that
\begin{eqnarray*}
  1/m < \beta < 2/m,~m/2 \le \alpha k < 1/\beta.  
\end{eqnarray*}
Using the aforementioned result and the painless non-orthogonal expansion, one can prove the following result. 
\begin{thm}\label{QN}
Let $m\ge 2$ and $\alpha, \beta >0.$ If $\alpha\beta<1,$ then there exists a $\gamma>0$ depending on $\alpha\beta$ such that $\mathcal{G}(Q_{m}(\gamma\cdot),\alpha,\beta)$ forms a frame for $L^2(\mathbb{R}).$ 
\end{thm}
\begin{figure}[H]
 \centering 
 \begin{subfigure}[b]{0.465\textwidth}
\centering
\includegraphics[width=\textwidth]{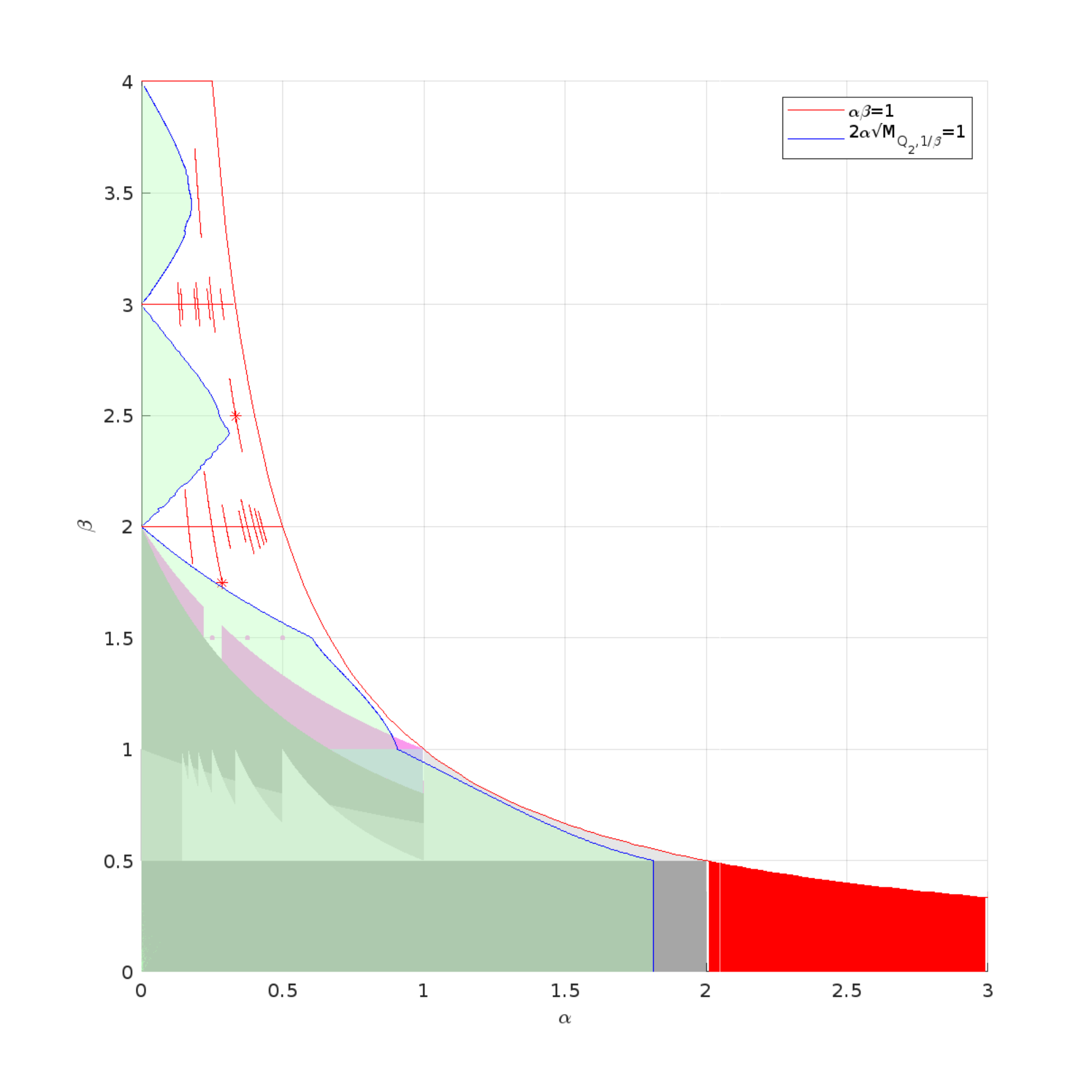}
\caption{}
\end{subfigure}
\begin{subfigure}[b]{0.49\textwidth}
         \centering
\includegraphics[width=\textwidth]{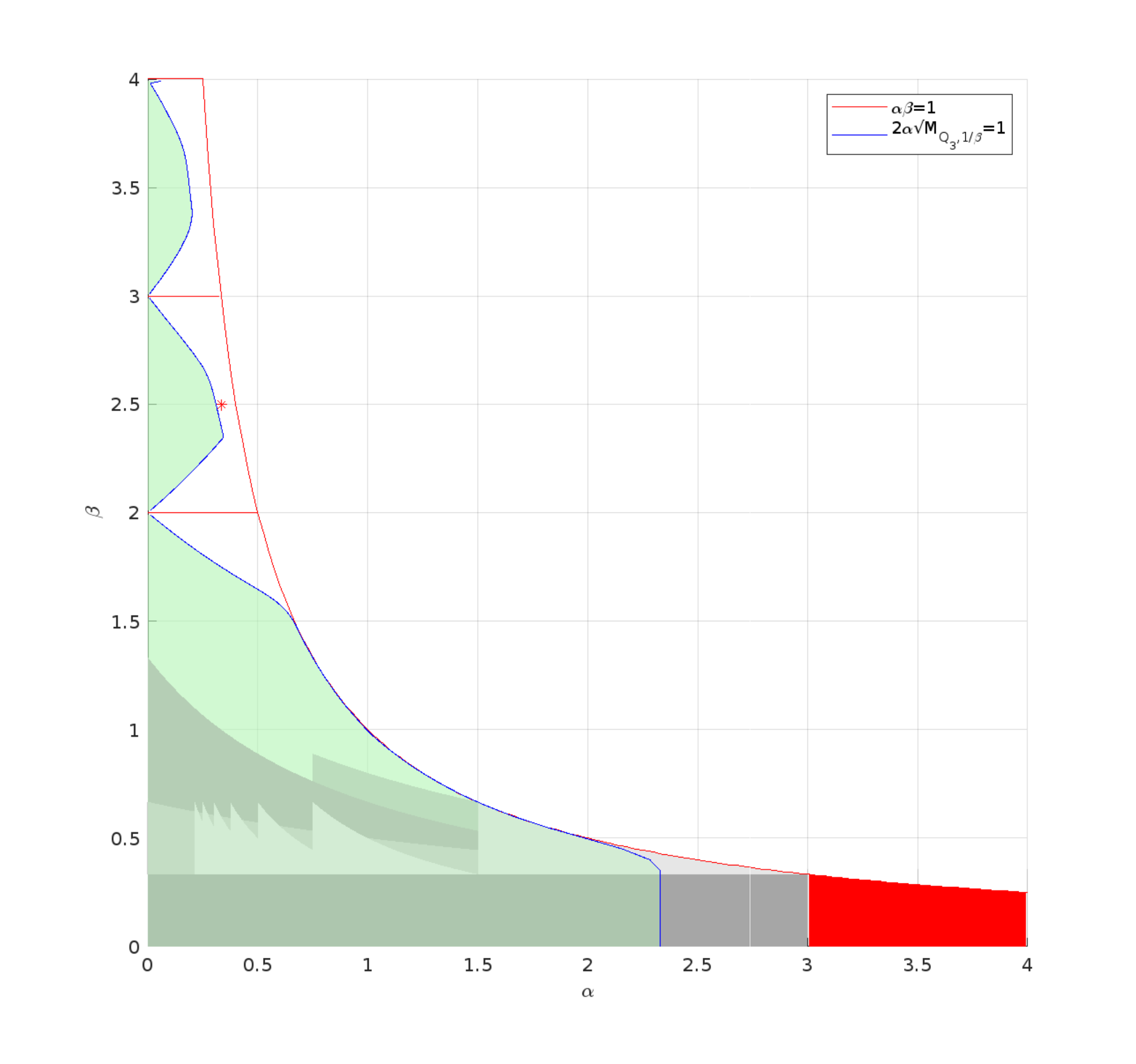}
\caption{}
\end{subfigure}
\caption{Sketches (A) and (B) represent some of the known results on the frame set for $Q_2$ and $Q_3$ respectively. The red color indicates the points $(\alpha, \beta)$ not to be in the frame set. The pink region for $Q_2$ was proved to be in $\mathcal{F}(Q_2)$ in \cite{ofsb}. The other shaded regions which are results from \cite{fscf, sign, olec, painless, counter} indicate the frame property. In each sketch, Theorem \ref{frame} asserts that the greenish region belongs to the frame set.}
\label{bsplinesfig}
\end{figure}
\section{Gabor frames with totally positive functions of type-I}
Consider the Gaussian function 
\begin{align*}
g(w):=e^{-\pi w^2},~w\in\mathbb{R}.
\end{align*}
To find the maximum of $B_{g_\gamma,1/\beta}(w)$, let us introduce the functions
$$H_{\rho}(w)=\sum_{n\in\mathbb{Z}}(w+n)^2e^{-\tfrac{2\pi}{\rho^2}(w+n)^2},~ F_{\rho}(w)=\sum_{n\in\mathbb{Z}}e^{-\tfrac{2\pi}{\rho^2}(w+n)^2},~G_{\rho}(w)=\dfrac{H_{\rho}(w)}{F_{\rho}(w)},$$ and the Jacobi theta function 
$$\Theta(w; t)=\sum\limits_{k\in\mathbb{Z}}e^{-\pi k^2 t}e^{2\pi ikw};~(w,t)\in\mathbb{R}\times\mathbb{R}_+.$$ It is easy to show that $F_{\rho}(w)$, $H_{\rho}(w)$, and $G_{\rho}(w)$ are 1-periodic positive symmetric functions about $w=\tfrac{1}{2}.$
\begin{lem}\cite{jacobi}\label{positiveQ}
Let $Q(w;t)=-\dfrac{\tfrac{\partial}{\partial w}\Theta(w; t)}{\sin 2\pi w}$ and $t>0$ fixed. Then $Q(w;t)$ is an even function of $w$ with period 1 and all its values are positive. Furthermore, for any $w\in\mathbb{R}$ and any $t> 0$, we have
$$A(t) \le Q(w; t) \le B(t),$$
where 
\begin{eqnarray*}
A(t)=
\begin{cases}
t^{-3/2}e^{-\tfrac{\pi}{4t}}, &0< t< 1,\\
(1-\tfrac{1}{3000})4\pi e^{-\pi t},&1\le t,
\end{cases} \text{and}~B(t)=
\begin{cases}
t^{-3/2},&0< t<1,\\
(1+\tfrac{1}{3000})4\pi e^{-\pi t},&1\le t.
\end{cases}
\end{eqnarray*}
\end{lem}
\begin{lem}\label{posF}
For each $0<\rho<2$, $G_{\rho}(w)$ attains its maximum at $w=\tfrac{1}{2}$ in $[0,1]$.
\end{lem}
\begin{proof}
 We first show that $F_{\rho}(w)$ is a decreasing function in $[0,\tfrac{1}{2}].$ Since the Fourier transform of $e^{-\tfrac{2\pi t^2}{\rho^2}}$ is $\dfrac{\rho}{\sqrt{2}}e^{-\tfrac{\pi\rho^2w^2}{2}},$ 
\begin{eqnarray}\label{Frho}
F_{\rho}(w)=\dfrac{\rho}{\sqrt{2}}\sum\limits_{n\in\mathbb{Z}}e^{-\tfrac{\pi\rho^2n^2}{2}}e^{2\pi i nw}=\dfrac{\rho}{\sqrt{2}}\Theta(w; \tfrac{\rho^2}{2}),
\end{eqnarray}
using the Poisson summation formula. Differentiating \eqref{Frho} with respect to $w$, we get
\begin{align}
F^{\prime}_{\rho}(w)&=-\dfrac{4\pi}{\rho^2}\sum\limits_{n\in\mathbb{Z}}(w+n)e^{-\tfrac{2\pi}{\rho^2}(w+n)^2}\nonumber\\
&=-\dfrac{4\pi\rho}{\sqrt{2}}\sum\limits_{n\in\mathbb{N}}ne^{-\tfrac{\pi\rho^2n^2}{2}}\sin{2\pi nw}\nonumber\\
&=-\dfrac{\rho}{\sqrt{2}}Q(w; \tfrac{\rho^2}{2})\sin{2\pi w}.\label{f'}
\end{align}
Since $\sin{2\pi w}\ge0$ in $[0,\tfrac{1}{2}]$, it follows from Lemma \ref{positiveQ} that $F^{\prime}_{\rho}(w)\le0$ in $[0,\tfrac{1}{2}].$ 
Thus $F_{\rho}(w)$ is a decreasing function in $[0,\tfrac{1}{2}].$ 
To prove our result, we show that $G_{\rho}(w)$ is an increasing function in $[0,\tfrac{1}{2}]$. It is enough to show that $\Psi(w):=H_{\rho}^{\prime}(w)F_{\rho}(w)-H_{\rho}(w)F^{\prime}_{\rho}(w)\ge0$ in $[0,\tfrac{1}{2}].$ 
We first write
\begin{align*}
\Psi(w)={}&\sum\limits_{j\in\mathbb{Z}}
\left[2(w+j)-\tfrac{4\pi}{\rho^2}(w+j)^3\right]e^{-\tfrac{2\pi}{\rho^2}(w+j)^2}\sum\limits_{k\in\mathbb{Z}}e^{-\tfrac{2\pi}{\rho^2}(w+k)^2}\nonumber\\&\hspace{3.5cm}+\dfrac{4\pi}{\rho^2}\sum\limits_{j\in\mathbb{Z}}(w+j)^2e^{-\tfrac{2\pi}{\rho^2}(w+j)^2}\sum\limits_{k\in\mathbb{Z}}(w+k)e^{-\tfrac{2\pi}{\rho^2}(w+k)^2}\\
={}&2\sum\limits_{j,k\in\mathbb{Z}}
\left[(w+j)-\tfrac{2\pi}{\rho^2}(w+j)^3+\tfrac{2\pi}{\rho^2}(w+j)^2(w+k)\right]e^{-\tfrac{2\pi}{\rho^2}\left\{(w+j)^2+(w+k)^2\right\}}\\
={}&2\sum\limits_{j,k\in\mathbb{Z}}
\left[(w+j)+\tfrac{2\pi}{\rho^2}(w+j)^2(k-j)\right]e^{-\tfrac{2\pi}{\rho^2}\{(w+j)^2+(w+k)^2\}}\nonumber\\
={}&2I_1(w)+\dfrac{4\pi}{\rho^2}I_2(w),
\end{align*}
where 
$$I_1(w):=\sum\limits_{j\in\mathbb{Z}}(w+j)e^{-\tfrac{2\pi}{\rho^2}(w+j)^2}\sum\limits_{k\in\mathbb{Z}}e^{-\tfrac{2\pi}{\rho^2}(w+k)^2}$$
and 
$$I_2(w):=\sum\limits_{j,l\in\mathbb{Z}}l(w+j)^2 e^{-\tfrac{2\pi}{\rho^2}\left\{(w+j)^2+(w+j+l)^2\right\}}.$$
Since $F_\rho(w)$ is a decreasing function in $[0,\tfrac{1}{2}],$ $I_1(w)\ge0$ from \eqref{f'}. To show $I_2(w)$ is nonnegative in $[0,\tfrac{1}{2}]$, we rewrite $I_2(w)$
\begin{align}
&=\sum\limits_{j\in\mathbb{Z}}\sum\limits_{l\in\mathbb{N}}l(w+j)^2 e^{-\tfrac{2\pi}{\rho^2}\{(w+j)^2+(w+j+l)^2\}}-\sum\limits_{j\in\mathbb{Z}}\sum\limits_{l\in\mathbb{N}}l(w+j)^2 e^{-\tfrac{2\pi}{\rho^2}\{(w+j)^2+(w+j-l)^2\}}\nonumber\\
&=\sum\limits_{j\in\mathbb{Z}}\sum\limits_{l\in\mathbb{N}}l\left[(w+j)^2-(w+j+l)^2\right]e^{-\tfrac{2\pi}{\rho^2}\{(w+j)^2+(w+j+l)^2\}}\nonumber\\
&=\sum\limits_{l\in\mathbb{N}}l^2\sum\limits_{j\in\mathbb{Z}}\left[-2(w+j)-l\right]e^{-\tfrac{2\pi}{\rho^2}\{(w+j)^2+(w+j+l)^2\}}.\label{I_2}
\end{align}
If $f_l(x)=-(2x+l)e^{-\tfrac{2\pi}{\rho^2}\left[x^2+(x+l)^2\right]}$, then $\widehat{f_l}(j)=(-1)^{jl}\dfrac{i\rho^3}{4}j e^{-\tfrac{\pi l^2}{\rho^2}}e^{-\tfrac{
\pi \rho^2 j^2}{4}},~j\in\mathbb{Z}.$
Using the Poisson summation formula and \eqref{f'}, \eqref{I_2} becomes
\begin{align}\label{positiveI2}
I_2(w)={}&\sum\limits_{l\in\mathbb{N}}l^2\sum\limits_{j\in\mathbb{Z}}(-1)^{jl}\dfrac{i\rho^3}{4}j e^{-\tfrac{\pi l^2}{\rho^2}}e^{-\tfrac{
\pi \rho^2 j^2}{4}}e^{2\pi ijw}\nonumber\\
={}&\dfrac{\rho^3}{2}\sum\limits_{l\in\mathbb{N}}l^2e^{-\tfrac{\pi l^2}{\rho^2}}\sum\limits_{j\in\mathbb{N}}(-1)^{jl+1}j e^{-\tfrac{
\pi \rho^2 j^2}{4}}\sin{2\pi jw}\nonumber\\
={}&\dfrac{\rho^3}{2}\Big[\sum\limits_{l\in\mathbb{N}}(2l-1)^2e^{-\tfrac{\pi (2l-1)^2}{\rho^2}}\sum\limits_{j\in\mathbb{N}}je^{-\tfrac{\pi \rho^2 j^2}{4}}\sin{2\pi j(\tfrac{1}{2}-w)}\nonumber\\
&{}\hspace{4cm}-\sum\limits_{l\in\mathbb{N}}(2l)^2e^{-\tfrac{\pi(2l)^2}{\rho^2}}\sum\limits_{j\in\mathbb{N}}je^{-\tfrac{\pi \rho^2 j^2}{4}}\sin{2\pi jw}\Big]\nonumber\\
={}&\dfrac{\rho^3}{8\pi}\left[h_1Q(\tfrac{1}{2}-w; \tfrac{\rho^2}{4})\sin{2\pi (\tfrac{1}{2}-w)}-h_2Q(w; \tfrac{\rho^2}{4})\sin{2\pi w}\right]\nonumber\\
={}&\dfrac{\rho^3}{8\pi}\sin{2\pi w}\left[h_1Q(\tfrac{1}{2}-w;\tfrac{\rho^2}{4})-h_2Q(w; \tfrac{\rho^2}{4})\right],
\end{align}
where $h_1=\sum\limits_{l\in\mathbb{N}}(2l-1)^2e^{-\tfrac{\pi}{\rho^2}(2l-1)^2}$ and $h_2=\sum\limits_{l\in\mathbb{N}}(2l)^2e^{-\tfrac{\pi}{\rho^2}(2l)^2}.$ It follows from Lemma \ref{positiveQ} that for $0<\rho<2$,
\begin{align*}
h_1 Q(\tfrac{1}{2}-w; \tfrac{\rho^2}{4})-h_2Q(w; \tfrac{\rho^2}{4})&\ge \left(\dfrac{\rho}{2}\right)^{-3}h_2\left(e^{-\tfrac{\pi}{\rho^2}}\dfrac{h_1}{h_2}-1\right)\\
&=\left(\dfrac{\rho}{2}\right)^{-3}h_2\left(\dfrac{e^{\tfrac{2\pi}{\rho^2}}+\sum\limits_{l=2}^{\infty}(2l-1)^2e^{-\tfrac{\pi}{\rho^2}\left((2l-1)^2-3\right)}}{4+\sum\limits_{l=2}^{\infty}(2l)^2e^{-\tfrac{\pi}{\rho^2}\left((2l)^2-4\right)}}-1\right)
\end{align*}
\begin{align*}
&>\left(\dfrac{\rho}{2}\right)^{-3}h_2\left(\dfrac{e^{\tfrac{\pi}{2}}}{4+\sum\limits_{l=2}^{\infty}(2l)^2e^{-\tfrac{\pi}{4}\left((2l)^2-4\right)}}-1\right)\\
 &>\left(\dfrac{\rho}{2}\right)^{-3}h_2\left(\dfrac{e^{\tfrac{\pi}{2}}}{4\left(1+\sum\limits_{l=2}^{\infty}l^2e^{-\pi (l^2-1)}\right)}-1\right)>0,
\end{align*}
using the fact that $\sum\limits_{l=2}^{\infty}l^2e^{-\pi (l^2-1)}<\tfrac{1}{3000}$. Thus $I_2(w)>0$ in $[0,\tfrac{1}{2}]$ from \eqref{positiveI2}. By symmetricity, $G_{\rho}(w)$ attains its maximum at $w=\tfrac{1}{2}$ in $[0,1]$ for all $0< \rho <2.$
\end{proof}
\begin{lem}\label{common}
If $\phi$ is a totally positive function of type-I, then $M_{\phi,h}\le M_{\widehat{g},h}.$
\end{lem}
\begin{proof}
Let $\phi\in L^1(\mathbb{R})$ such that its Fourier transform is of the form
\begin{eqnarray*}
\widehat{\phi}(w)=g(w)\prod\limits_{j=1}^{N}\left(1+2\pi i v_jw\right)^{-1}e^{-2\pi iv_jw},
\end{eqnarray*}
where $g(w)=e^{-\pi w^2}$, $v_j\in\mathbb{R}$, $N\in\mathbb{N}\cup\{\infty\}$, and $0<\sum_{j}v_j^2<\infty$. We first prove our result when $N$ is finite. For each $\mu=0,1,\dots,N,$ we define 
 $$\widehat{\psi_{\mu}}(w):=g(w)\prod\limits_{j=0}^{\mu}\left(1+2\pi i v_jw\right)^{-1}e^{-2\pi iv_jw}~\text{with}~v_0=0,~f_{\mu,h}(w)=|\widehat{\psi_{\mu}}(\tfrac{w}{h})|^2,$$ and $\widetilde{B}_{\mu,h}(w)=\dfrac{\sum\limits_{n\in\mathbb{Z}}(w+n)^2f_{\mu,h}(w+n)}{\sum\limits_{n\in\mathbb{Z}}f_{\mu,h}(w+n)}.$  Notice that $M_{\psi_{\mu},h}=\dfrac{1}{h^2} \esssup\limits_{w\in[0,1]}\widetilde{B}_{\mu,h}(w).$
For computational convenience, we write $f_{\mu,h}\equiv f_\mu,$ $\widetilde{B}_{\mu,h}\equiv \widetilde{B}_{\mu},$ and $s_j=\dfrac{4\pi^2v_j^2}{h^2}$. To prove our result, it is enough to show that
\begin{align*}
\widetilde{B}_{\mu}(w)-\widetilde{B}_{\mu+1}(w)
=\dfrac{\sum\limits_{n,k\in\mathbb{Z}}\left\{(w+k)^2-(w+n)^2\right\}f_{\mu}(w+k)f_{\mu+1}(w+n)}{\sum\limits_{k\in\mathbb{Z}}f_{\mu}(w+k)\sum\limits_{n\in\mathbb{Z}}f_{\mu+1}(w+n)}\ge 0,\
\end{align*}
for $ w\in[0,1]$ and $\mu=0,1,\dots,N.$ Since
$$\hspace{-3cm}\sum\limits_{n,k\in\mathbb{Z}}\left\{(w+k)^2-(w+n)^2\right\}f_{\mu}(w+k)f_{\mu+1}(w+n)$$
\begin{align*}
={}&\sum\limits_{n,k\in\mathbb{Z}}\{2w(k-n)+k^2-n^2\}f_{\mu}(w+k)f_{\mu+1}(w+n)\\
={}&\sum\limits_{n,k\in\mathbb{Z}} k(2w+2n+k)f_{\mu}(w+n+k)\dfrac{f_{\mu}(w+n)}{1+s_{\mu+1}(w+n)^2}
\end{align*}
\begin{align*}
\begin{split}
={}&\sum\limits_{n\in\mathbb{Z}}\sum\limits_{k\in\mathbb{N}}k(2w+2n+k)f_{\mu}(w+n+k)\dfrac{f_{\mu}(w+n)}{1+s_{\mu+1}(w+n)^2}\\&\hspace{3cm}-\sum\limits_{n\in\mathbb{Z}}\sum\limits_{k\in\mathbb{N}}k(2w+2n-k)f_{\mu}(w+n-k)\dfrac{f_{\mu}(w+n)}{1+s_{\mu+1}(w+n)^2}\\
={}&\sum\limits_{n\in\mathbb{Z}}\sum\limits_{k\in\mathbb{N}}k(2w+2n+k)f_{\mu}(w+n+k)\dfrac{f_{\mu}(w+n)}{1+s_{\mu+1}(w+n)^2}\\&\hspace{3cm}-\sum\limits_{n\in\mathbb{Z}}\sum\limits_{k\in\mathbb{N}}k(2w+2n+k)f_{\mu}(w+n)\dfrac{f_{\mu}(w+n+k)}{1+s_{\mu+1}(w+n+k)^2}\\
={}&\sum\limits_{n\in\mathbb{Z}}\sum\limits_{k\in\mathbb{N}}\dfrac{s_{\mu+1}k^2(2w+2n+k)^2f_{\mu}(w+n+k)f_{\mu}(w+n)}{\left(1+s_{\mu+1}(w+n)^2\right)\left(1+s_{\mu+1}(w+n+k)^2\right)}\ge 0,
\end{split}
\end{align*}
$\widetilde{B}_{\mu+1}(w)\le \widetilde{B}_{\mu}(w)$ for $w\in[0,1].$ As $\mu\rightarrow\infty$, we have
$$\lim\limits_{\mu\rightarrow\infty}\dfrac{\sum\limits_{n\in\mathbb{Z}}(w+n)^2f_{\mu}(w+n)}{\sum\limits_{n\in\mathbb{Z}}f_{\mu}(w+n)}=\dfrac{\sum\limits_{n\in\mathbb{Z}}(w+n)^2|\widehat{\phi}(\tfrac{w+n}{h})|^2}{\sum\limits_{n\in\mathbb{Z}}|\widehat{\phi}(\tfrac{w+n}{h})|^2},$$
by the Lebesgue dominated convergence theorem. For $N=\infty$, our result now follows from the monotonicity property of $\widetilde{B}_{\mu}(w)$.
\end{proof} 
\begin{thm}\label{B1gammalem}
If $\phi$ is a totally positive function of type-I, then $\lim\limits_{\gamma\rightarrow0}M_{\phi_\gamma,1/\beta}=\dfrac{\beta^2}{4}.$
Consequently, if $\alpha\beta<1,$ then there exists a $\gamma>0$ depending on $\alpha\beta$ such that $\mathcal{G}(\phi_\gamma,\alpha,\beta)$ forms a frame for $L^2(\mathbb{R}).$ 
\end{thm}
\begin{proof}
Notice that $M_{g_\gamma,1/\beta}=\beta^2\esssup\limits_{w\in[0,1]}\dfrac{\sum\limits_{j\in\mathbb{Z}}(w+j)^2e^{-\tfrac{2\pi}{\rho^2}(w+j)^2}}{\sum\limits_{j\in\mathbb{Z}}e^{-\tfrac{2\pi}{\rho^2}(w+j)^2}},$ where $\rho=\dfrac{\gamma}{\beta}.$ From Lemma \ref{posF}, we have for $0< \rho <2$
\begin{align*}
M_{g_\gamma,1/\beta}
&=\dfrac{\beta^2\sum\limits_{j\in\mathbb{Z}}(\tfrac{1}{2}+j)^2e^{-\tfrac{2\pi}{\rho^2}\left(\tfrac{1}{2}+j\right)^2}}{\sum\limits_{j\in\mathbb{Z}}e^{-\tfrac{2\pi}{\rho^2}\left(\tfrac{1}{2}+j\right)^2}}\\
&\le\dfrac{\beta^2\sum\limits_{j\in\mathbb{Z}}(\tfrac{1}{2}+j)^2e^{-\tfrac{2\pi}{\rho^2}\left(\tfrac{1}{2}+j\right)^2}}{2e^{-\tfrac{2\pi}{\rho^2}\tfrac{1}{4}}}\\
&=\dfrac{\beta^2}{2}\sum\limits_{j\in\mathbb{Z}}(\tfrac{1}{2}+j)^2e^{-\tfrac{2\pi}{\rho^2}\left[\left(\tfrac{1}{2}+j\right)^2-\tfrac{1}{4}\right]}
\end{align*}
\begin{align*}
&=\dfrac{\beta^2}{2}\sum\limits_{j\in\mathbb{Z}}(\tfrac{1}{2}+j)^2e^{-\tfrac{2\pi}{\rho^2}j(j+1)}\\
&=\dfrac{\beta^2}{4}+\dfrac{\beta^2}{2}\sum\limits_{j\in\mathbb{Z}\setminus\{0,-1\}}(\tfrac{1}{2}+j)^2e^{-\tfrac{2\pi}{\rho^2}j(j+1)}.
\end{align*}
It is easy to show that $\lim\limits_{\rho\rightarrow0}\sum\limits_{j\in\mathbb{Z}\setminus\{0,-1\}}(\tfrac{1}{2}+j)^2e^{-\tfrac{2\pi}{\rho^2}j(j+1)}=0$
by the Lebesgue dominated convergence theorem. Therefore, $M_{g_\gamma,1/\beta}\le \tfrac{\beta^2}{4}$ as $\gamma\rightarrow 0.$ We have already shown that $M_{\phi_\gamma,1/\beta}\ge \tfrac{\beta^2}{4}.$ Hence $\lim\limits_{\gamma\rightarrow 0} M_{\phi_\gamma,1/\beta}=\tfrac{\beta^2}{4}$ from Lemma \ref{common}. 

If $\alpha\beta<1,$ then $\alpha\beta=1-\epsilon,$ for some $\epsilon>0.$ Since
$\lim\limits_{\gamma\rightarrow0}2\alpha\sqrt{M_{\phi_\gamma,1/\beta}}=\alpha\beta,$
there exists a $\delta>0$ depending on $\alpha\beta$ such that 
$|2\alpha\sqrt{M_{\phi_\gamma,1/\beta}}-\alpha\beta|<\epsilon, ~\text{whenever}~0<\gamma\le\delta,$
which implies 
\begin{eqnarray*}
2\alpha\sqrt{M_{\phi_\gamma,1/\beta}}<\alpha\beta +\epsilon=1,~\text{whenever}~0<\gamma\le\delta.
\end{eqnarray*}
The desired result now follows from Theorem \ref{frame}. 
\end{proof}
Based on Theorem \ref{frame}, we have drawn a region belonging to the frame set for the Gaussian function using MATLAB. Our region is very close to the complete frame set. Please refer to Figure \ref{gaussianfig}. 
\begin{figure}[H]
 \centering \begin{subfigure}[b]{0.49\textwidth}
\centering
\includegraphics[width=\textwidth]{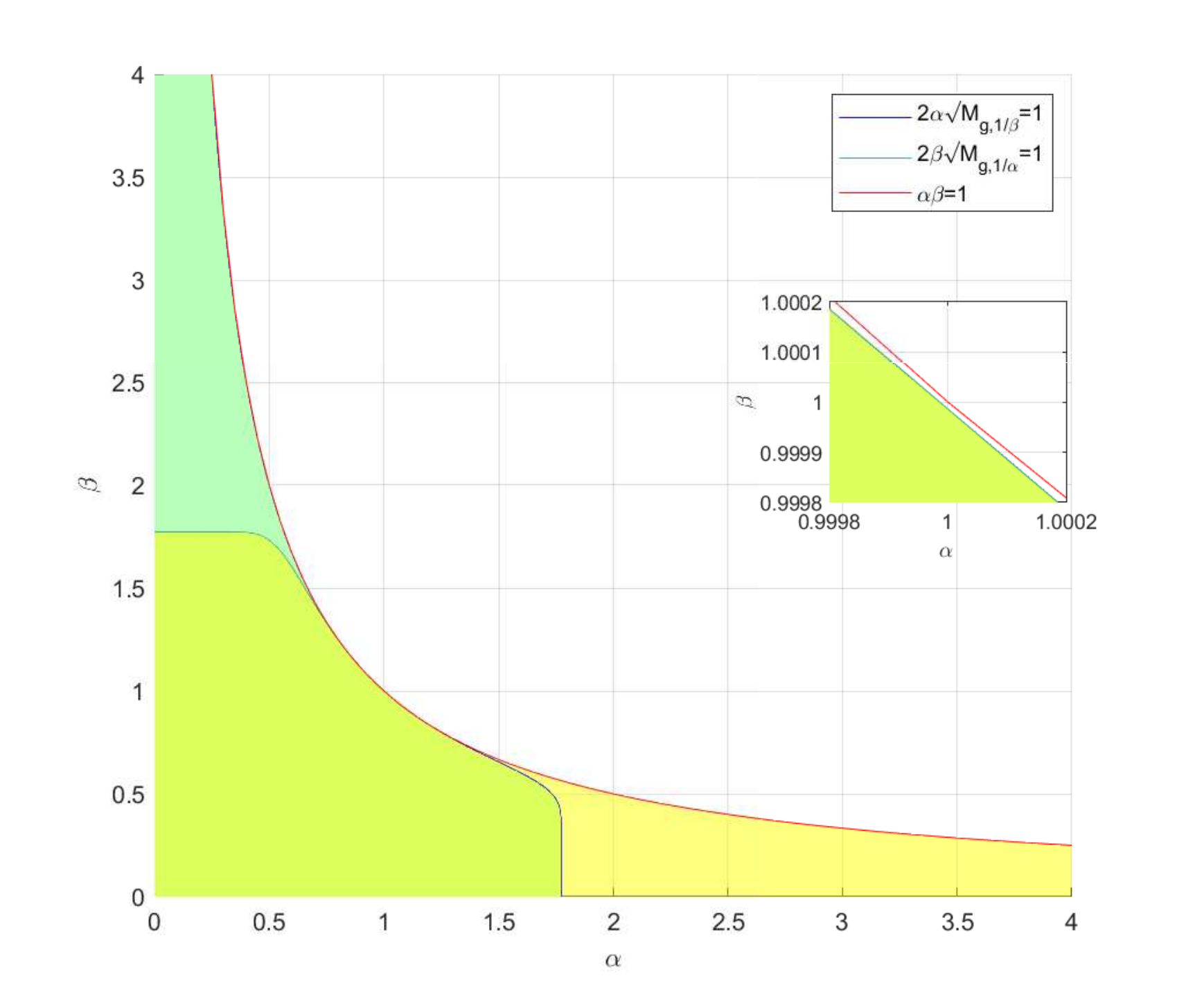}
\caption{}
\end{subfigure}
\begin{subfigure}[b]{0.49\textwidth}
         \centering
\includegraphics[width=\textwidth]{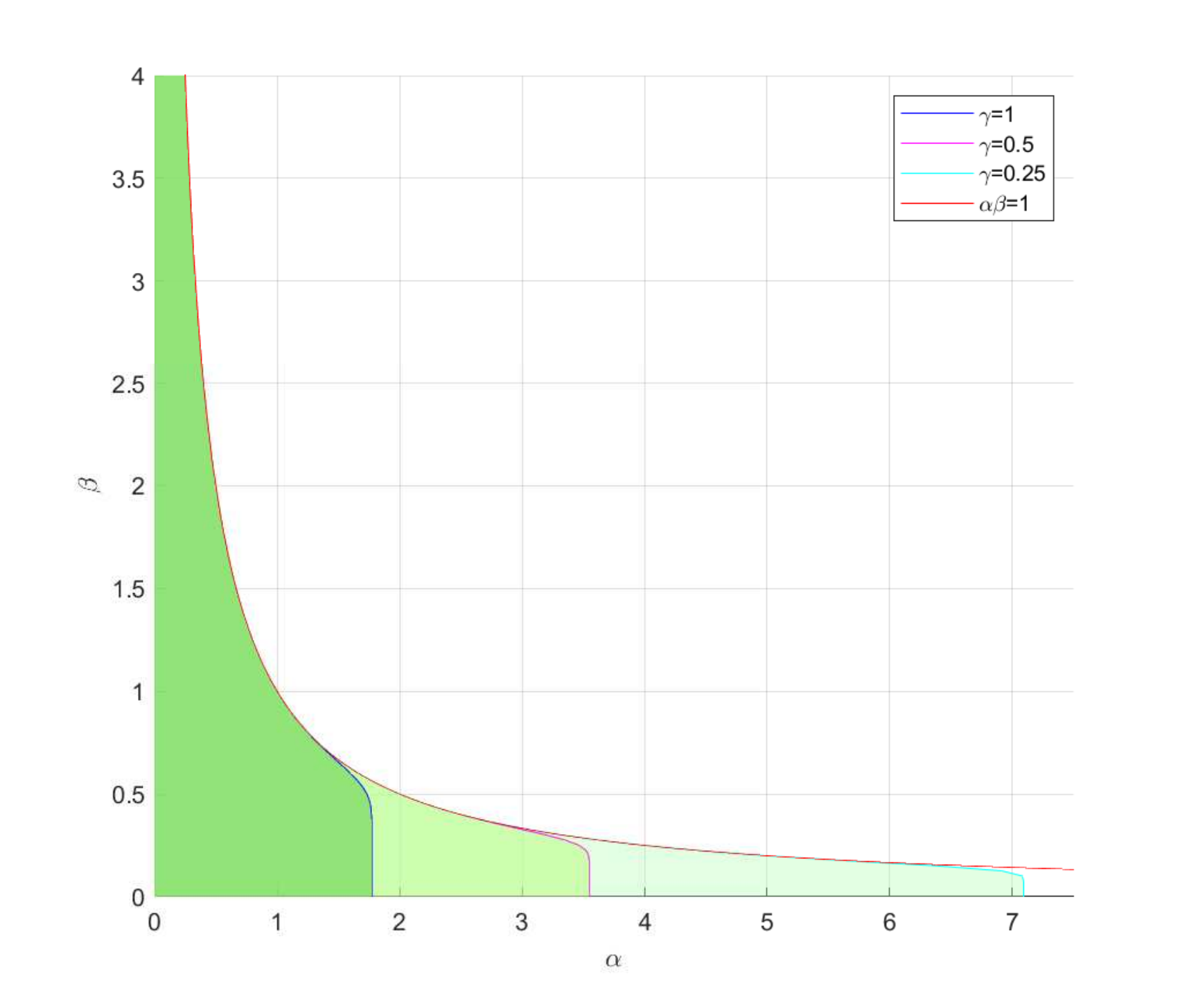}
\caption{}
\end{subfigure}
\caption{In the sketch (A) the greenish color represents the frame region of $g(x)=e^{-\pi x^2}$ which is obtained from the condition $0<\alpha<\left({2\sqrt{M_{g,1/\beta}}}\right)^{-1}$ and yellow region is obtained using the self Fourier transform property of $g$. The subgraph explains that our frame region is not optimal. Sketch $($B$)$ demonstrates the validity of Theorem \ref{B1gammalem} for totally positive functions of type-I.}
\label{gaussianfig}
\end{figure}
\section{Gabor frames with functions of type-II}
Let 
$$g_1(w)=e^{-|w|}~\text{and}~g_2(w)=\dfrac{1}{(1+2\pi iaw)(1+2\pi ibw)},~w\in\mathbb{R},$$
where $a$ and $b$ are nonzero real numbers. Let $$\mathsf{H_{\rho}}(w)=\sum\limits_{n\in\mathbb{Z}}(w+n)^2\left|g_1\left(\rho(w+n)\right)\right|^2, \mathsf{F_{\rho}}(w)=\sum\limits_{n\in\mathbb{Z}}\left|g_1\left(\rho(w+n)\right)\right|^2,$$ and $\mathsf{G_{\rho}}(w)=\dfrac{H_{\rho}(w)}{F_{\rho}(w)}.$ Similarly, we can define $\widetilde{H}_\rho,$ $\widetilde{F}_\rho$ and $\widetilde{G}_\rho$ for $g_2.$
\begin{lem}\label{positiveI1}
For each $\rho\ge0.74$, $f_{\rho}(w)=e^{4\rho (1-2w)}-\dfrac{we^{-2\rho}+1-w}{w-we^{-2\rho}+e^{-2\rho}}$ is nonnegative in $[0, \tfrac{1}{2}]$.
\end{lem}
\begin{proof}
If $h_{\rho}(w)=\dfrac{we^{-2\rho}+1-w}{w-we^{-2\rho}+e^{-2\rho}},$ then $h_{\rho}^{\prime}(w)=-\dfrac{1-e^{-4\rho}}{\left(w-we^{-2\rho}+e^{-2\rho}\right)^2}<0.$ Therefore $h_{\rho}(w)$ is a decreasing function in $[0,\tfrac{1}{2}]$ and hence $h_\rho(w)\le e^{2\rho}$ in $[0,\tfrac{1}{2}].$ Consequently, $f_{\rho}(w)\ge e^{2\rho}\left(e^{2\rho(1-4w)}-1\right)\ge 0$ in $[0,\tfrac{1}{4}].$ To show $f_{\rho}(w)$ is nonnegative in $ (\tfrac{1}{4},\tfrac{1}{2}],$ we first write
\begin{eqnarray*}
f_{\rho}(w)
=\dfrac{w\left(e^{4\rho (1-2w)}-e^{-2\rho}\right)+(1-w)(e^{2\rho (1-4w)}-1)}{w+e^{-2\rho}(1-w)}.
\end{eqnarray*}
Since $w+e^{-2\rho}(1-w)>0$ for $w\in (\tfrac{1}{4},\tfrac{1}{2}],$ it is enough to show that 
\begin{eqnarray}\label{phirho}
\phi_{\rho}(w):=\dfrac{w}{1-w}\dfrac{e^{2\rho (2-4w)}-e^{-2\rho}}{1-e^{2\rho (2-4w)}e^{-2\rho}}\ge 1,~w\in (\tfrac{1}{4},\tfrac{1}{2}].
\end{eqnarray}
Differentiating \eqref{phirho} with respect to $w$, we get 
$$\phi_{\rho}^{\prime}(w)=-\dfrac{\left(8\rho(1-e^{4\rho})w^2+8\rho(e^{4\rho}-1)w-(e^{4\rho}+1)\right)e^{2\rho(2-4w)}+e^{4\rho(2-4w)+2\rho}+e^{2\rho}}{(w-1)^2(e^{2\rho(2-4w)}-e^{2\rho})^2}.$$
By calculus technique again, we can easily show that $\psi_{\rho}(w):=8\rho(1-e^{4\rho})w^2+8\rho(e^{4\rho}-1)w-(e^{4\rho}+1)$ is an increasing function in $[ \tfrac{1}{4},\tfrac{1}{2}]$  and $h(\rho):=\rho-\tfrac{2}{3}\tfrac{e^{4\rho}+1}{e^{4\rho}-1}$ is an increasing function in $(0,\infty)$. Since $h(0.74)\ge0$,
$$\psi_{\rho}(w)\ge\psi_{\rho}(\tfrac{1}{4})=\tfrac{3\rho}{2}(e^{4\rho}-1)-(e^{4\rho}+1)=\tfrac{3}{2}(e^{4\rho}-1)h(\rho)\ge0,$$ for all $\rho\ge 0.74,$ which implies that $\phi_{\rho}(w)$ is a decreasing function in $(\tfrac{1}{4},\tfrac{1}{2}]$. Since $\phi_{\rho}(\tfrac{1}{2})=1$, \eqref{phirho} holds for all $\rho\ge 0.74$. 
\end{proof}
\begin{lem}\label{positiveG}
For each $ \rho \ge0.74$, $\mathsf{G}_{\rho}(w)$ attains its maximum at $w=\tfrac{1}{2}$ in $[0,1]$.
\end{lem}
\begin{proof}
To prove our result, it is enough to show that  $\mathsf{G}_{\rho}(w)$ is an increasing function in $[0,\tfrac{1}{2}].$ Let us define $$I_1(w):=\sum\limits_{l=0}^{\infty}\sum\limits_{n=0}^{\infty}(w+l)e^{-2\rho\left(2w+n+l\right)}+\sum\limits_{l=1}^{\infty}\sum\limits_{n=1}^{\infty}(w-l)e^{-2\rho\left(-2w+n+l\right)}$$ and 
\begin{multline*}
I_2(w):=\sum\limits_{l=1}^{\infty}\sum\limits_{n=0}^{\infty}\left[w-l+2\rho(w-l)^2\right]e^{-2\rho\left(n+l\right)}\\+\sum\limits_{l=0}^{\infty}\sum\limits_{n=1}^{\infty}\left[w+l-2\rho(w+l)^2\right]e^{-2\rho\left(n+l\right)}.    
\end{multline*}
Then 
\begin{align}
I_1(w)&=e^{-4\rho w}\sum\limits_{l=0}^{\infty}(w+l)e^{-2\rho l}\sum\limits_{n=0}^{\infty}e^{-2\rho n}+e^{4\rho w}\sum\limits_{l=1}^{\infty}(w-l)e^{-2\rho l}\sum\limits_{n=1}^{\infty}e^{-2\rho n}\nonumber\\
&=\left(\dfrac{w}{1-e^{-2\rho}}+\dfrac{e^{-2\rho}}{(1-e^{-2\rho})^2}\right)\dfrac{e^{-4\rho w}}{1-e^{-2\rho}}+\left(\dfrac{we^{-2\rho}}{1-e^{-2\rho}}-\dfrac{e^{-2\rho}}{(1-e^{-2\rho})^2}\right)\dfrac{e^{-2\rho(1-2w)}}{1-e^{-2\rho}}\nonumber\\
&=\dfrac{e^{-4\rho w}}{(1-e^{-2\rho})^3}\left(w-we^{-2\rho}+e^{-2\rho}\right)+\dfrac{e^{-4\rho (1-w)}}{(1-e^{-2\rho})^3}\left(w-we^{-2\rho}-1\right)\nonumber\\
&=\dfrac{e^{-4\rho (1-w)}\left(w-we^{-2\rho}+e^{-2\rho}\right)}{\left(1-e^{-2\rho}\right)^3}\left\{e^{4\rho (1-2w)}-\dfrac{we^{-2\rho}+1-w}{w-we^{-2\rho}+e^{-2\rho}}\right\}>0\nonumber
\end{align}
for every $\rho\ge 0.74$ from Lemma \ref{positiveI1}. Similarly, $I_2(w)$
\begin{align*}
\begin{split}
I_2(w)={}& \sum\limits_{l=1}^{\infty}\left[w-l+2\rho(w-l)^2\right]e^{-2\rho l}+\sum\limits_{n=1}^{\infty}\left(w-2\rho w^2\right)e^{-2\rho n}\\&\hspace{7cm}+\sum\limits_{l=1}^{\infty}\sum\limits_{n=1}^{\infty}(2w-8\rho lw)e^{-2\rho (l+n)}\\
={}&\sum\limits_{l=1}^{\infty}\left(2w-l+2\rho l^2-4\rho lw\right)e^{-2\rho l}+\sum\limits_{l=1}^{\infty}\sum\limits_{n=1}^{\infty}(2w-8\rho lw)e^{-2\rho (l+n)}\\
={}&\dfrac{2we^{-2\rho}}{1-e^{-2\rho}}-(4\rho w+1)\dfrac{e^{-2\rho}}{(1-e^{-2\rho})^2}+2\rho\dfrac{e^{-2\rho}+e^{-4\rho}}{(1-e^{-2\rho})^3}\\&\hspace{5.5cm}+\left(\dfrac{2we^{-2\rho}}{1-e^{-2\rho}}-\dfrac{8\rho we^{-2\rho}}{(1-e^{-2\rho})^2}\right)\dfrac{e^{-2\rho}}{1-e^{-2\rho}}\\
={}&\dfrac{e^{-2\rho}}{\left(1-e^{-2\rho}\right)^3}\left[2we^{-4\rho}+\left(2\rho+4\rho w+1-4w\right)e^{-2\rho}+2\rho-4\rho w-1+2w\right]\\&\hspace{5cm}+\dfrac{e^{-2\rho}}{\left(1-e^{-2\rho}\right)^3}\left(2we^{-2\rho}-2we^{-4\rho}-8\rho we^{-2\rho}\right)\\
={}&\dfrac{e^{-2\rho}}{\left(1-e^{-2\rho}\right)^3}\left[2w(1-e^{-2\rho}-2\rho-2\rho e^{-2\rho})+2\rho e^{-2\rho}+e^{-2\rho}+2\rho-1\right]>0
\end{split}
\end{align*}
using the fact that $(2\rho+1)e^{-2\rho}+2\rho-1>0$ for $\rho>0$ and $w\le\tfrac{1}{2}.$ 
Since 
\newline
$$\hspace{-8.6cm}\mathsf{H}_{\rho}^{\prime}(w)\mathsf{F}_{\rho}(w)-\mathsf{H}_{\rho}(w)\mathsf{F}^{\prime}_{\rho}(w)$$
\begin{align*}
\begin{split}
={}&\sum\limits_{l\in\mathbb{Z}}\left[2(w+l)-\dfrac{2\rho (w+l)^3}{|w+l|}\right]e^{-2\rho|w+l|}\sum\limits_{n\in\mathbb{Z}}e^{-2\rho|w+n|}\\&\hspace{4cm}+2\rho\sum\limits_{l\in\mathbb{Z}}(w+l)^2e^{-2\rho|w+l|}\sum\limits_{n\in\mathbb{Z}}\dfrac{w+n}{|w+n|}e^{-2\rho|w+n|}\\
={}&2\sum\limits_{l\in\mathbb{Z}}\sum\limits_{n\in\mathbb{Z}}\left(w+l-\rho\dfrac{(w+l)^3}{|w+l|}+\rho(w+l)^2\dfrac{w+n}{|w+n|}\right)e^{-2\rho\left(|w+l|+|w+n|\right)}\\
={}&2\sum\limits_{l\in\mathbb{Z}}\sum\limits_{n\in\mathbb{Z}}\left\{w+l+\rho (w+l)^2\left(\sgn{(w+n)}-\sgn{(w+l)}\right)\right\}e^{-2\rho\left(|w+l|+|w+n|\right)}\\
={}&2\Bigg\{\sum\limits_{l=0}^{\infty}\sum\limits_{n=0}^{\infty}\left(w+l\right)e^{-2\rho\left(2w+n+l\right)}+\sum\limits_{l=1}^{\infty}\sum\limits_{n=0}^{\infty}\left[w-l+2\rho (w-l)^2\right]e^{-2\rho\left(n+l\right)}\\&+\sum\limits_{l=0}^{\infty}\sum\limits_{n=1}^{\infty}\left[w+l-2\rho (w+l)^2\right]e^{-2\rho\left(n+l\right)}+\sum\limits_{l=1}^{\infty}\sum\limits_{n=1}^{\infty}\left(w-l\right)e^{-2\rho\left(-2w+n+l\right)}\Bigg\}\\
={}&2(I_1(w)+I_2(w))>0,
\end{split}
\end{align*}
$\mathsf{G}_{\rho}(w)$ is an increasing function in $[0,\tfrac{1}{2}]$ for each $\rho\ge 0.74$.
\end{proof}
\begin{lem}\label{positiveF}
For each $\rho>0$, $\widetilde{G_{\rho}}(w)$ attains its maximum at $w=\tfrac{1}{2}$ in $[0,1]$. 
\end{lem}
\begin{proof}
To prove our result, we make use of the following identities
$$\sum\limits_{n\in\mathbb{Z}}\dfrac{1}{w-n}=\pi\cot\pi w,~w\notin\mathbb{Z},$$
and
$$\cot(x+iy)-\cot(x-iy)=\dfrac{2i\sinh2y}{\cos2x-\cosh2y}.$$
We first assume that $a\neq \pm b.$ Let $c=\tfrac{\rho}{a}$ and $d=\tfrac{\rho}{b}$.
By applying the partial fraction expansion, we have
\begin{align*}
\widetilde{H_\rho}(w)={}&\sum\limits_{n\in\mathbb{Z}}\frac{c^2d^2(w+n)^2}{(c^2+4\pi^2(w+n)^2)(d^2+4\pi^2(w+n)^2)}\\
={}&\dfrac{c^2d^2}{4\pi^2(d^2-c^2)}\sum\limits_{n\in\mathbb{Z}}\left[\dfrac{d^2}{d^2+4\pi^2(w+n)^2}-\dfrac{c^2}{c^2+4\pi^2(w+n)^2}\right]\\
={}&\dfrac{c^2d^2}{8\pi^2(d^2-c^2)}\sum\limits_{n\in\mathbb{Z}}\left[\tfrac{d}{d+2\pi i(w+n)}+\tfrac{d}{d-{2\pi i}(w+n)}-\tfrac{c}{c+{2\pi i}(w+n)}-\tfrac{c}{c-{2\pi i}(w+n)}\right] \\
={}&\dfrac{c^2d^2}{8\pi^2(d^2-c^2)}\Big[\dfrac{id}{2}\left\{\cot\left(\pi\left(w+\tfrac{id}{2\pi }\right)\right)-\cot\left(\pi\left(w-\tfrac{id}{2\pi}\right)\right)\right\}\\&\hspace{4.5cm}-\dfrac{ic}{2}\left\{\cot\left(\pi\left(w+\tfrac{ic}{2\pi}\right)\right)-\cot\left(\pi\left(w-\tfrac{ic}{2\pi}\right)\right)\right\}\Big]\\
={}&\dfrac{c^2d^2}{8\pi^2(d^2-c^2)}\left[\dfrac{c\sinh{c}}{\cos{2\pi w}-\cosh{c}}-\dfrac{d\sinh{d}}{\cos{2\pi w}-\cosh{d}}\right]
\end{align*}
\begin{align*}
={}&\dfrac{c^2d^2}{8\pi^2(d^2-c^2)}\dfrac{d\cosh{c}\sinh{d}-c\cosh{d}\sinh{c}-\left(d\sinh{d}-c\sinh{c}\right)\cos{2\pi w}}{\left(\cos{2\pi w}-\cosh{d}\right)\left(\cos{2\pi w}-\cosh{c}\right)}
\end{align*}
and
\begin{align*}
\widetilde{F_\rho}(w)
={}&\sum\limits_{n\in\mathbb{Z}}\frac{c^2d^2}{(c^2+4\pi^2(w+n)^2)(d^2+4\pi^2(w+n)^2)}\\
={}&\dfrac{c^2d^2}{(d^2-c^2)}\sum\limits_{n\in\mathbb{Z}}\Big[\dfrac{1}{c^2+4\pi^2(w+n)^2}-\dfrac{1}{d^2+4\pi^2(w+n)^2}\Big]\\
={}&\dfrac{c^2d^2}{2(d^2-c^2)}\sum\limits_{n\in\mathbb{Z}}\Big[\tfrac{1}{c}\left(\tfrac{1}{c+2\pi i(w+n)}+\tfrac{1}{c-2\pi i(w+n)}\right)-\tfrac{1}{d}\left(\tfrac{1}{d+2\pi i(w+n)}+\tfrac{1}{d-2\pi i(w+n)}\right)\Big]\\
={}&\dfrac{c^2d^2}{4(d^2-c^2)}\Big[-\tfrac{i} {c}\cot\left(\pi\left(w-\tfrac{ic}{2\pi }\right)\right)+\tfrac{i}{c}\cot\left(\pi\left(w+\tfrac{ic}{2\pi }\right)\right)\\&\hspace{5.3cm}+\tfrac{i}{d}\cot\left(\pi\left(w-\tfrac{id}{2\pi }\right)\right)-\tfrac{i}{d}\cot\left(\pi\left(w+\tfrac{id}{2\pi }\right)\right)\Big]\\
={}&\dfrac{cd}{2(d^2-c^2)}\left[\dfrac{c\sinh{d}}{\cos{2\pi w}-\cosh{d}}-\dfrac{d\sinh{c}}{\cos{2\pi w}-\cosh{c}}\right]
\end{align*}
\begin{align*}
={}&\dfrac{cd}{2(d^2-c^2)}\dfrac{d\cosh{d}\sinh{c}-c\cosh{c}\sinh{d}+ \left(c\sinh{d}-d\sinh{c}\right)\cos{2\pi w}}{\left(\cos{2\pi w}-\cosh{d}\right)\left(\cos{2\pi w}-\cosh{c}\right)}.
\end{align*}
Let $A:=d\cosh{c}\sinh{d}-c\cosh{d}\sinh{c}$, $B:=d\sinh{d}-c\sinh{c}$, $C:=d\cosh{d}\sinh{c}-c\cosh{c}\sinh{d},$ and $D:=c\sinh{d}-d\sinh{c}.$
Then
\begin{eqnarray}\label{what}
\widetilde{G_\rho}(w)=\dfrac{cd}{4\pi^2}\dfrac{A-B\cos{2\pi w}}{C+D\cos{2\pi w}}.
\end{eqnarray}
Differentiating \eqref{what} with respect to $w$, we get
\begin{eqnarray*}
\widetilde{G_\rho}^{\prime}(w)=\dfrac{cd}{2\pi}\dfrac{(AD+BC)\sin{2\pi w}}{(C+D\cos{2\pi w})^2}.
\end{eqnarray*}
Notice that  $cd(AD+BC)=cd(d^2-c^2)\sinh c\sinh d\left(\cosh d-\cosh c\right)>0$.
Therefore $\widetilde{G_{\rho}}^{\prime}(w)\ge 0$ in $[0,\tfrac{1}{2}].$
By limiting argument, we can easily show that $\widetilde{G_{\rho}}(w)$ is an increasing function in $[0,\tfrac{1}{2}]$ for $a=\pm b$.
\end{proof}
 We can easily show that
$$M_{\widehat{{g_1}_\gamma},1/\alpha}=\alpha^2\esssup\limits_{w\in[0,1]}\dfrac{\sum\limits_{j\in\mathbb{Z}}(w+j)^2e^{-2\alpha\gamma|w+j|}}{\sum\limits_{j\in\mathbb{Z}}e^{-2\alpha\gamma|w+j|}}\longrightarrow\dfrac{\alpha^2}{4}~\text{as}~ \gamma\longrightarrow\infty.$$ One can observe that Lemma \ref{common} is still true after replacing the Gaussian function with the two-sided exponential. Arguing as in Theorem \ref{B1gammalem}, we can prove the following theorem using Lemma \ref{positiveG}. 
\begin{thm}\label{maintse}
If $\phi$ is a function of type-II, then $\lim\limits_{\gamma\rightarrow \infty} M_{\phi_\gamma,1/\alpha}=\dfrac{\alpha^2}{4}.$
Consequently, if $\alpha\beta<1,$ then there exists a $\gamma>0$ depending on $\alpha\beta$ such that $\mathcal{G}(\phi_\gamma,\alpha,\beta)$ forms a frame for $L^2(\mathbb{R}).$ 
\end{thm}
Again based on Theorem \ref{frame}, we have drawn a region belonging to the frame set for the two-sided exponential function using Lemmas \ref{positiveG} and \ref{positiveF}. Please refer to Figure \ref{tsefig}.
\begin{figure}[H]
 \centering \begin{subfigure}[b]{0.48\textwidth}
\centering
\includegraphics[width=\textwidth]{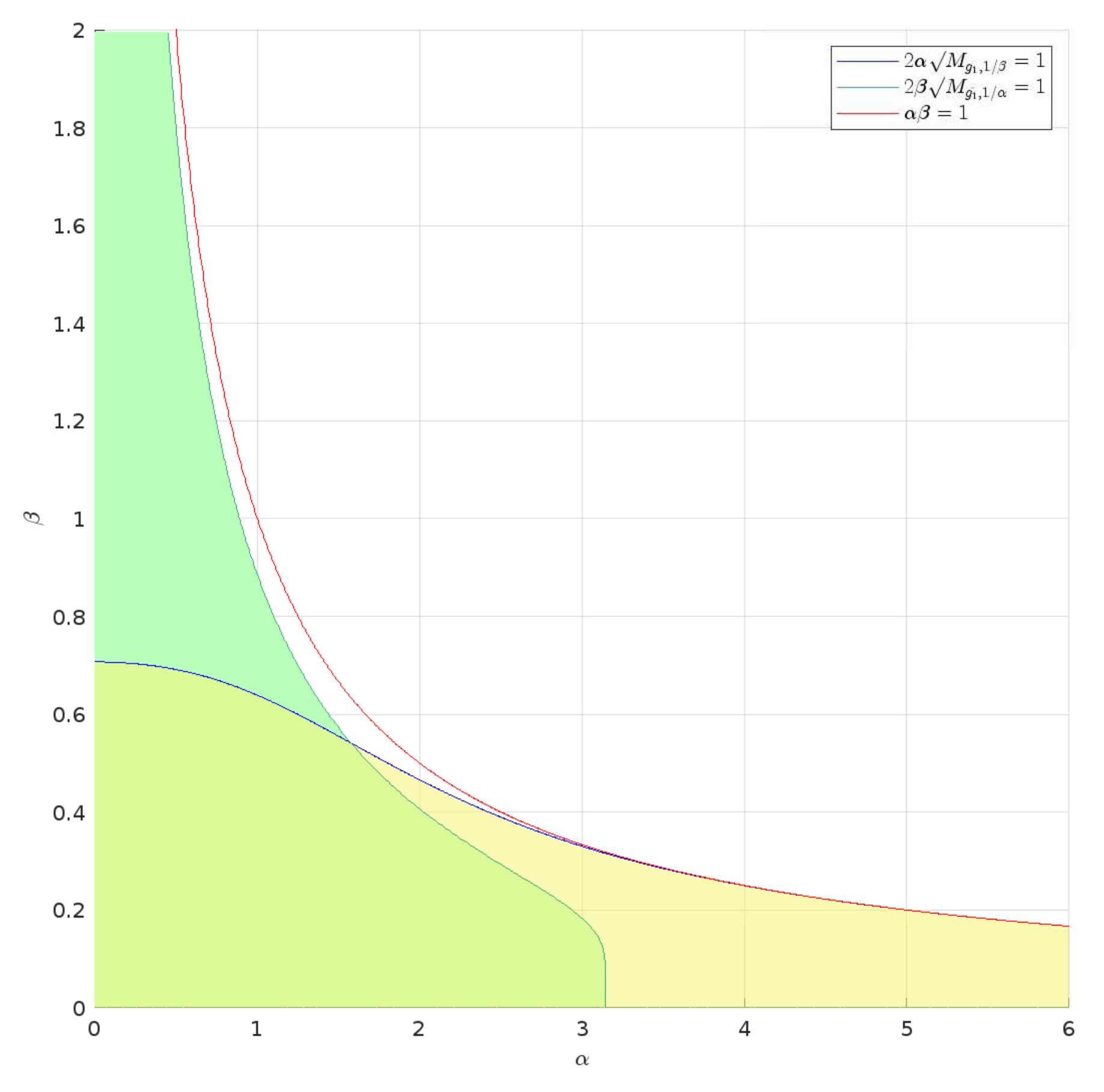}
\caption{}
\end{subfigure}
\hfill
\begin{subfigure}[b]{0.48\textwidth}
         \centering
\includegraphics[width=\textwidth]{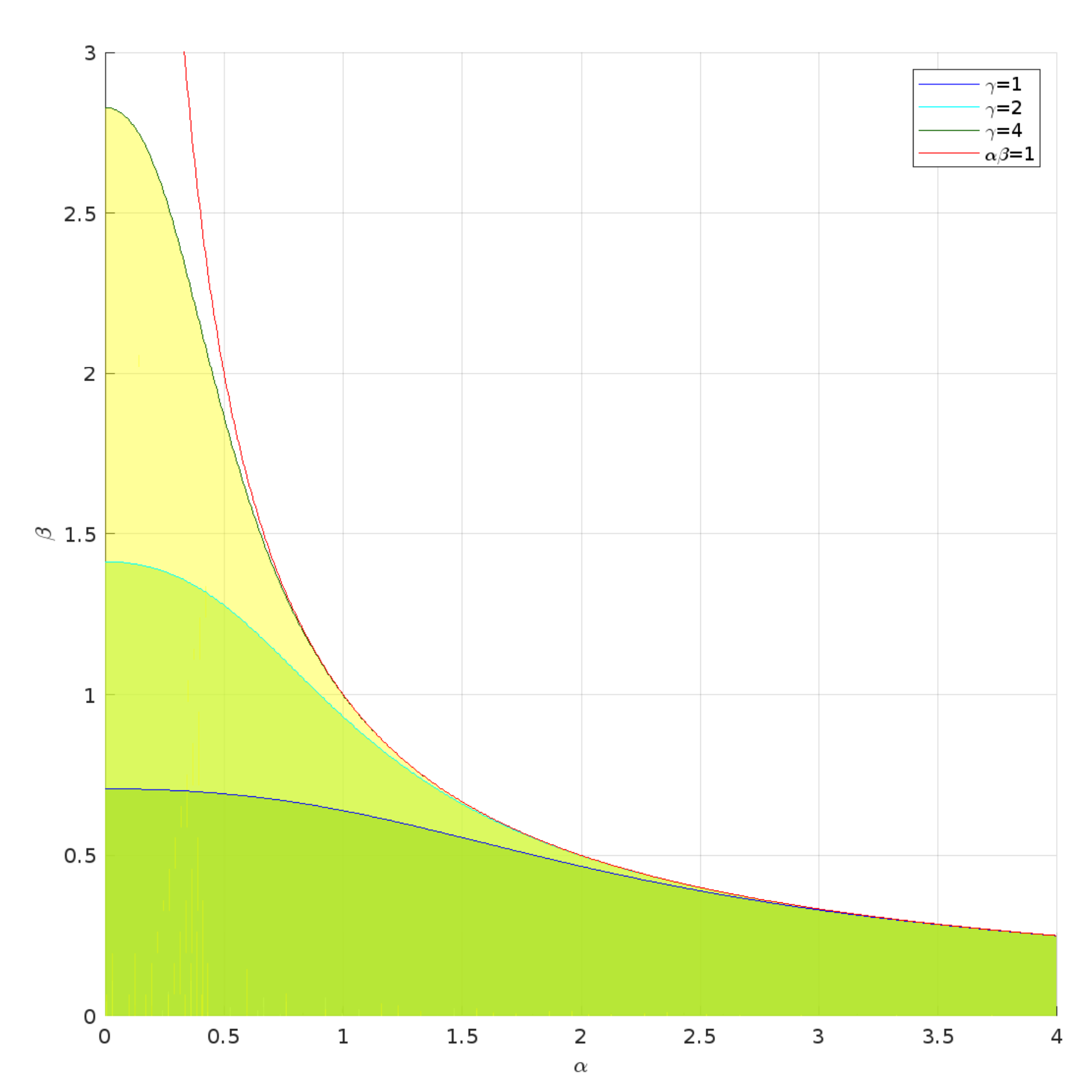}
\caption{}
\end{subfigure}
\caption{In the sketch (A) the greenish color represents the frame region of $g_1(x)=e^{-|x|}$ which is obtained from the condition $0<\alpha<(2\sqrt{M_{g_1,1/\beta}})^{-1}$ and the yellow region is obtained by from the condition $0<\beta<(2\sqrt{M_{\widehat{g_1},1/\alpha}})^{-1}$. Sketch (B) demonstrates the validity of Theorem \ref{maintse} for functions of type-II.}
\label{tsefig}
\end{figure}
\section{Gabor frames with Hermite functions}
In this section, we find a new region for Hermite functions $h_n$ defined in \eqref{hermitef}. Since $\sum_{j\in\mathbb{Z}}|h_n(w+\beta j)|^2$ is a positive continuous function for $w\in[0,\beta],$ there exists a constant $C>0$ such that $\essinf\limits_{w\in[0,\beta]}\sum_{j\in\mathbb{Z}}|h_n(w+\beta j)|^2\ge C.$ Therefore $h_n$ is a stable generator for $V_{1/\beta}(h_n).$ 
Let us define
$$\Delta_n(\alpha,\beta):=2\alpha\sqrt{M_{h_n,1/\beta}}~\text{and}~\widehat{\Delta_n}(\alpha,\beta):=2\beta\sqrt{M_{\widehat{h_n},1/\alpha}},~n=0,1,2,3,\dots.$$ 
Since the Fourier transform of  $h_n$ is $\widehat{h_n}(w)=(-i)^n h_n(w)$, it is clear that $\widehat{\Delta_n}(\alpha,\beta)=\Delta_n(\beta,\alpha).$
Consequently, if $\Delta_n(\alpha,\beta)<1$ or $\Delta_n(\beta,\alpha)<1$, then $\mathcal{G}(h_n,\alpha,\beta)$ forms a frame for $L^2(\mathbb{R})$ from Theorem \ref{frame}. We have already discussed the frame set of the Gaussian function $h_0(x)=e^{-\pi x^2}$ in the previous section. We used MATLAB to draw the graph of $\Delta_n(\alpha,\beta)$ and to determine the frame region of $h_n$, for $n=1,2,3,4,5.$ The authors in \cite{Lyu} proved that if $\alpha\beta=3/5,$ then $\mathcal{G}(h_1,\alpha,\beta)$ forms a frame in $L^2(\mathbb{R}).$ Lemvig \cite{hergabor} provided some points which do not belong in $\{(\alpha,\beta)\in \mathbb{R}^2_+:~\alpha\beta<1\}$ for $h_n$ with $n = 4,5,4m + 2, 4m + 3$,  $m\in\mathbb{N}_0$. Based on the aforementioned results, we expect that our frame
region for $h_{2n}$ might be close to the frame set. Please refer to Figure \ref{hermitefig}.
\begin{figure}[H]
     \centering
     \begin{subfigure}[b]{0.48\textwidth}
         \centering       \includegraphics[width=\textwidth]{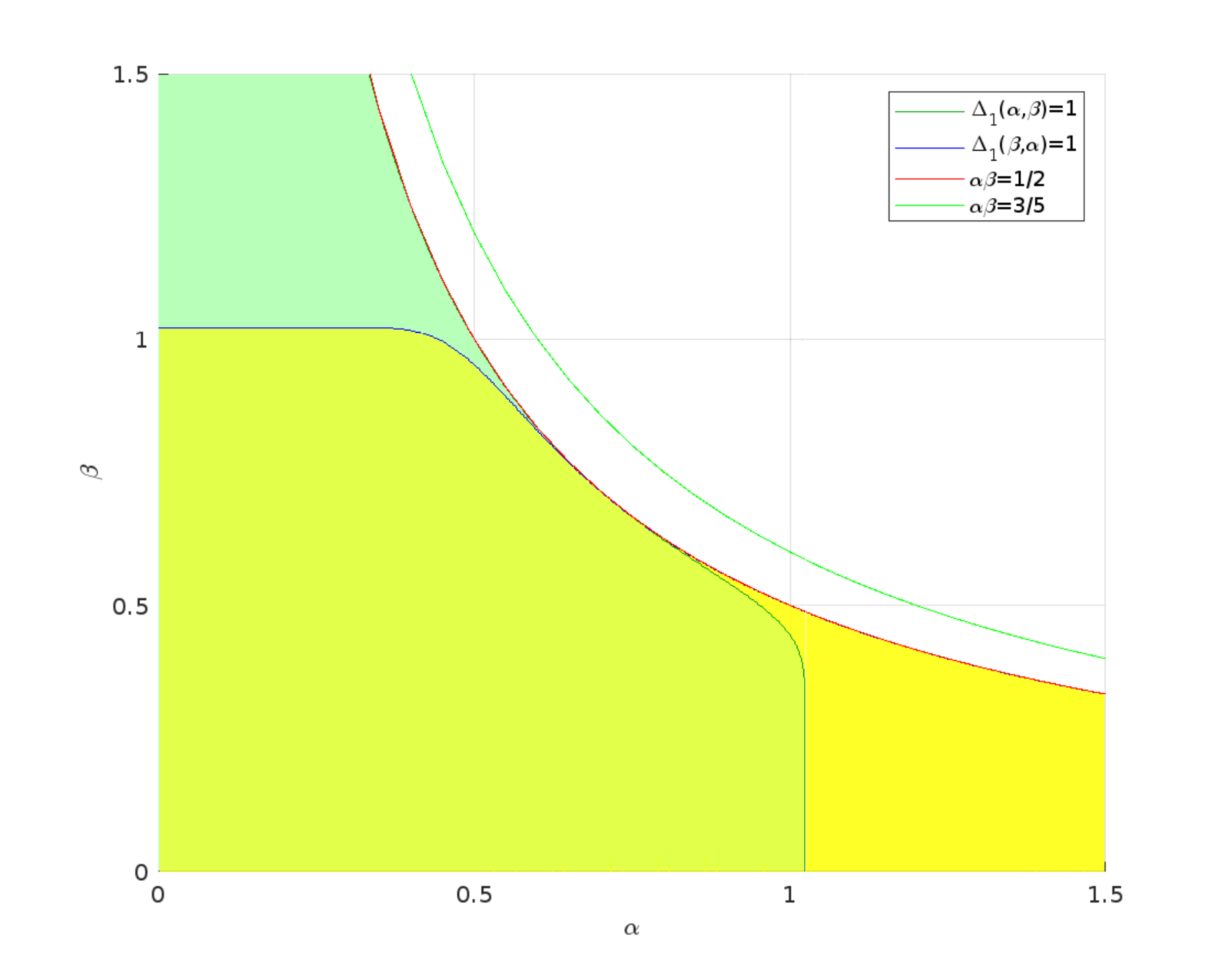}
         \caption{}
     \end{subfigure}
     \begin{subfigure}[b]{0.48\textwidth}
         \centering
         \includegraphics[width=\textwidth]{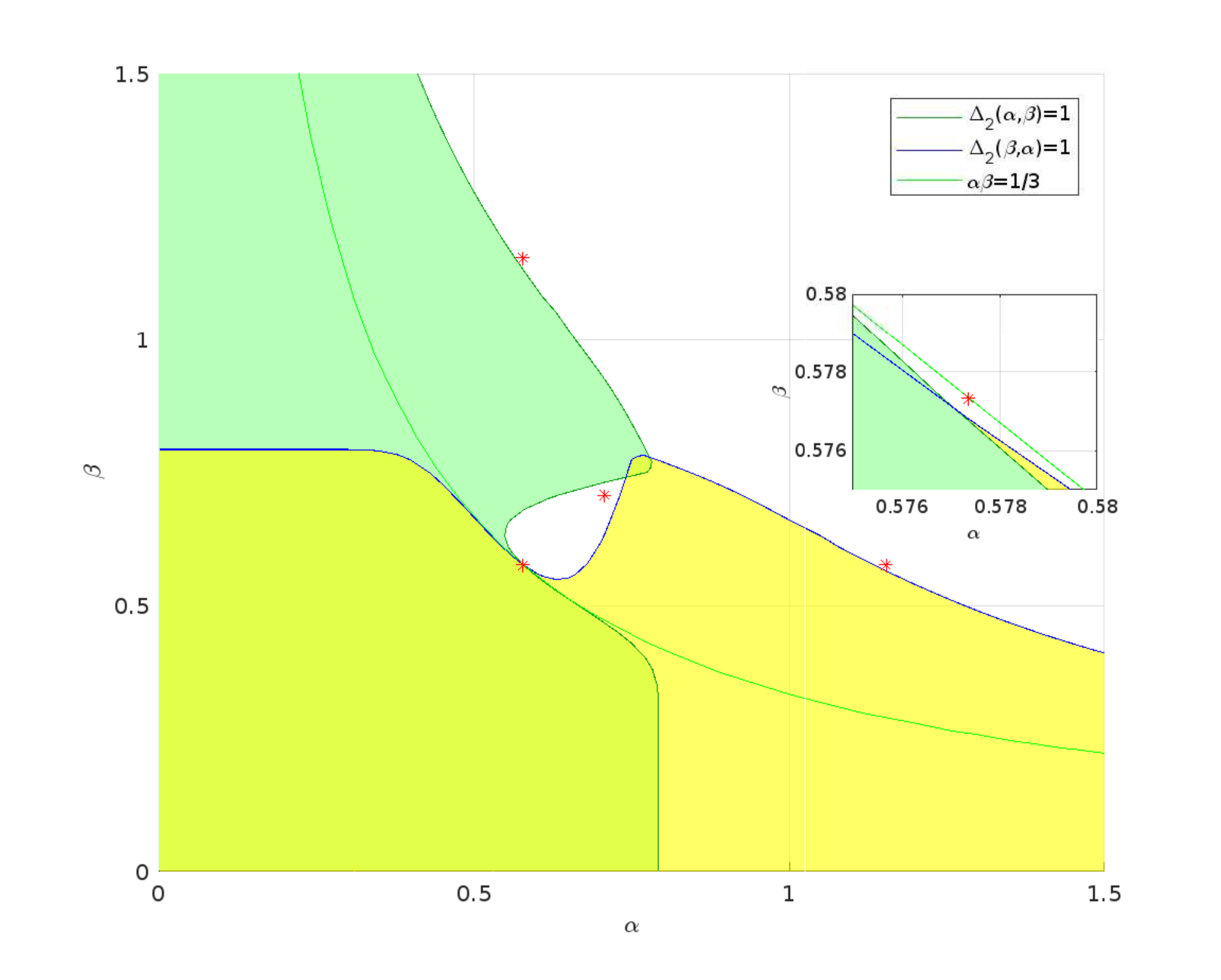}
         \caption{}
     \end{subfigure}
     \begin{subfigure}[b]{0.46\textwidth}
     \ContinuedFloat
         \centering
         \includegraphics[width=\textwidth]{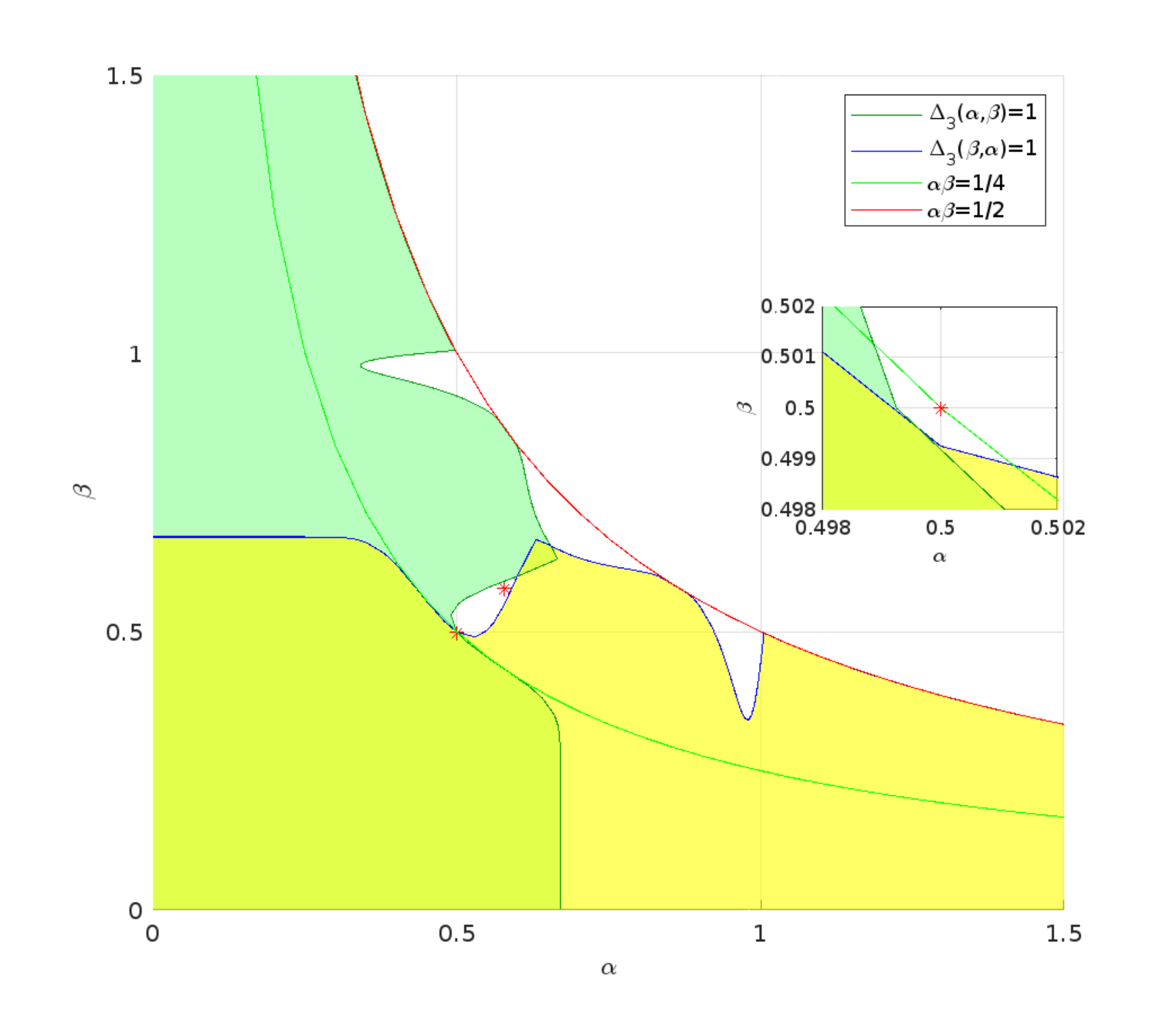}
         \caption{}
     \end{subfigure}
     \begin{subfigure}[b]{0.48\textwidth}
         \centering
         \includegraphics[width=\textwidth]{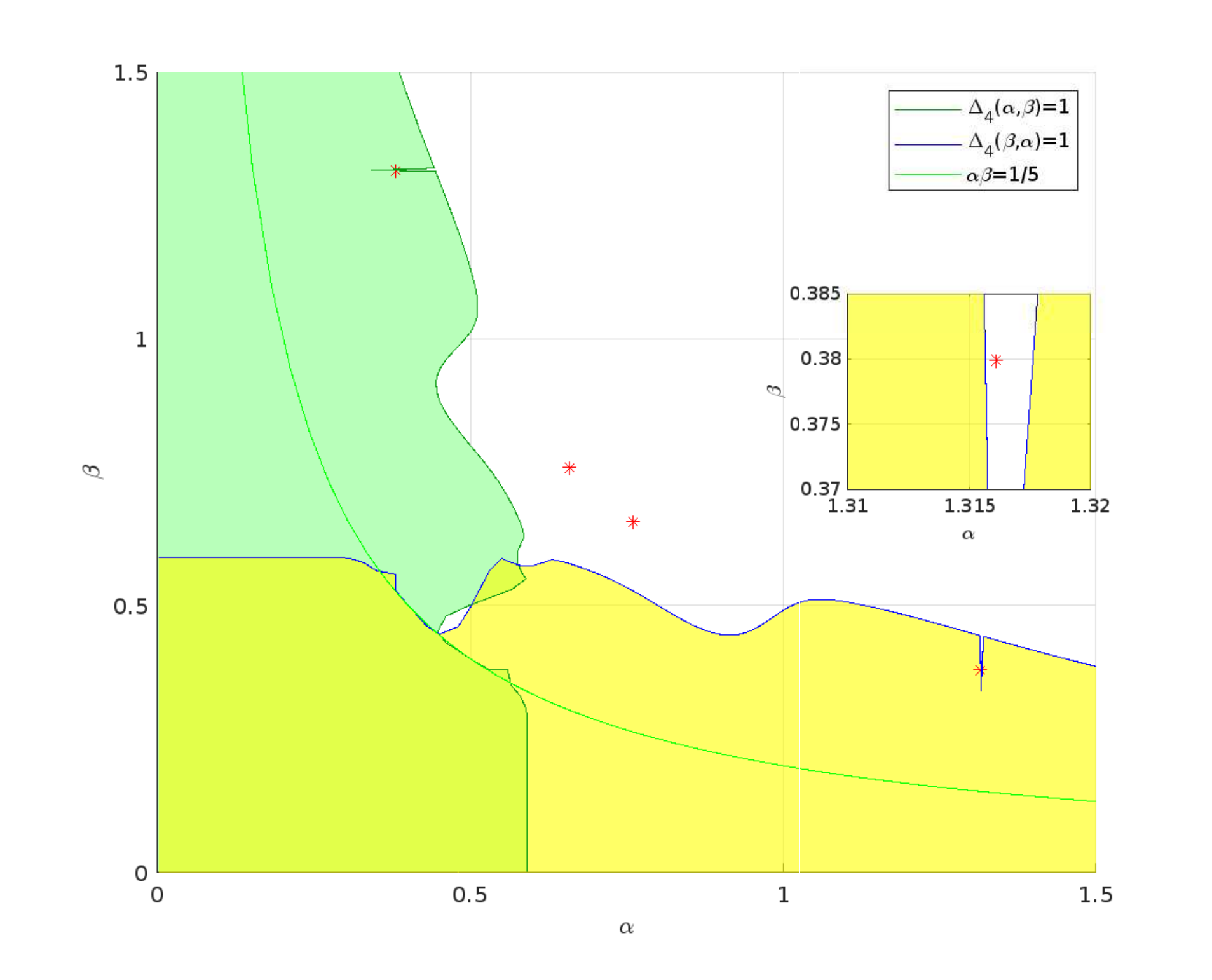}
         \caption{}
     \end{subfigure}
     \begin{subfigure}[b]{0.48\textwidth}
     \centering
         \centering
\includegraphics[width=\textwidth]{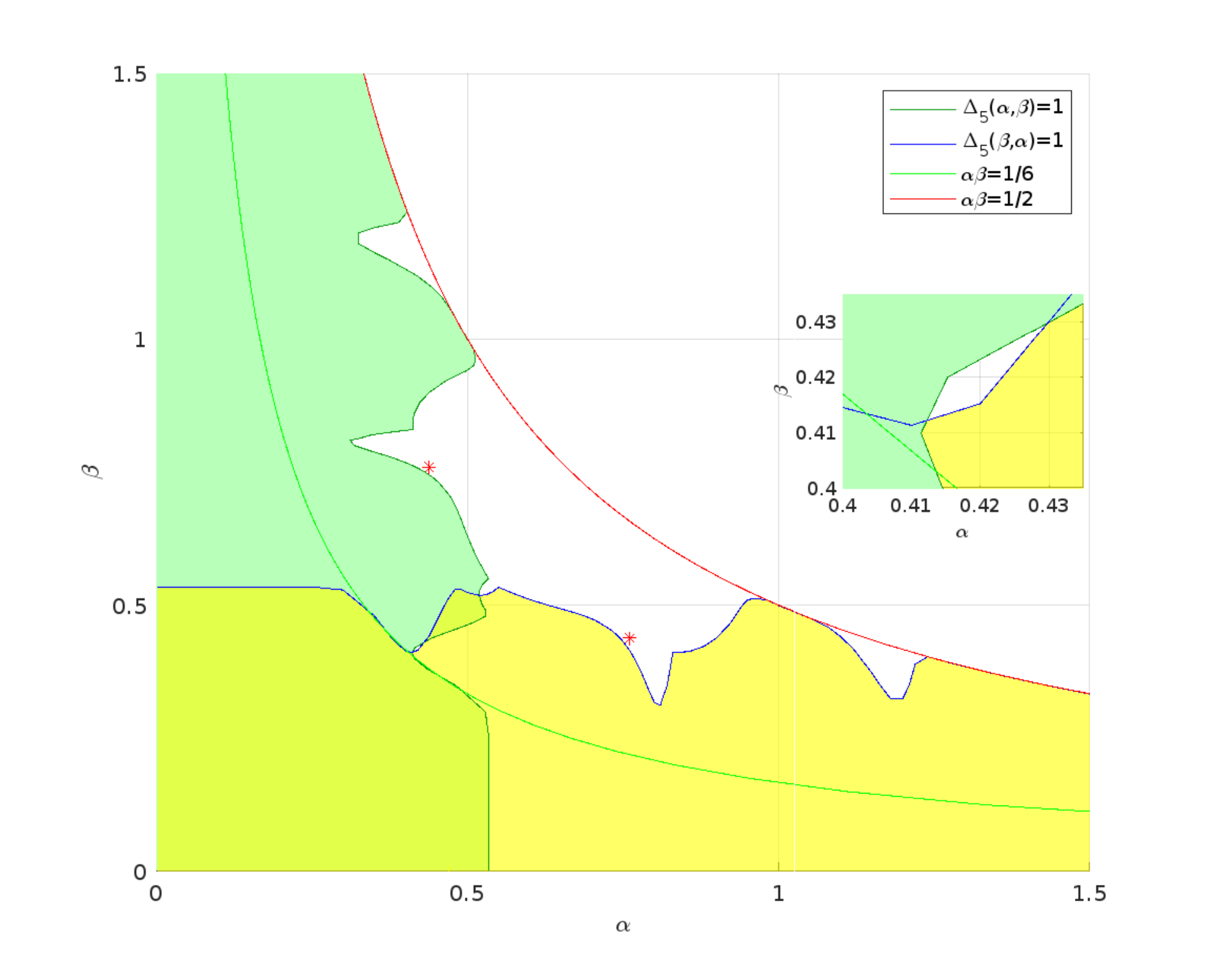}
         \caption{}  
         \end{subfigure}
\label{fig: three graphs}
\caption{(A)-(E) illustrate the regions belonging to the frame set for Hermite functions $h_1-h_5$ respectively. In sketch (A), the green line $\alpha\beta=3/5$ belongs to $\mathcal{F}(h_1)$ \cite{Lyu}. The author in \cite{hergabor} proved that `$\ast$' points do not lie in the frame set. The red line contains the points $(\alpha, \beta)$ which do not lie in the frame set \cite{mystery, Lyu}. The subgraphs in (B), (C), and (D) explain that our results did not contradict the results from \cite{hergabor}. The subgraph in (E) illustrates that our frame region does not include the white region.}
\label{hermitefig}
\end{figure}
\section{Appendix}
\captionsetup{font=footnotesize}
\begin{lstlisting}[language=MATLAB, caption= \underline{MATLAB code for the Budan-Fourier theorem and Sturm's algorithm}]
function [m] = sign_chng(x) %no of sign change in array x
c=0;
    for j=1:length(x)-1
        if ((x(j)>0)&&(x(j+1)<0))||((x(j)<0)&&(x(j+1)>0))
          c=c+1;
        else
            continue
        end
    end
    m=c;
end
%Y=coefficient array; m(n)=left(right) end point of [a,b] 
%Budan-Fourier theorem
function w = Budan_Fourier(Y,m,n)
    syms x
    P= poly2sym(Y, x)
    for i=1:length(Y)
       Vl(i)=subs(diff(P,i-1),x,m);
       Vr(i)=subs(diff(P,i-1),x,n);
    end
 fprintf('VFB(a):[%s]\n', join(string(Vl),','));
 fprintf('VFB(b):[%s]\n',join(string(Vr),','))
 V1=nonzeros(Vl); V2=nonzeros(Vr);
fprintf('V_F(%2.2f) is %d\n',m,sign_chng(V1))
fprintf('V_F(%2.2f) is %d\n',n,sign_chng(V2))
w=sign_chng(V1)-sign_chng(Vr);%maximum no of roots
end
%Sturm's algorithm
function w = Sturm(Y,m,n)
    syms x
    P= poly2sym(Y, x)
    p=subs(P, x, m)*subs(P, x, n);
    if p==0
        disp('Sturm method not possible')
        w=NaN;%cannot find roots
    else
      DP=diff(P);
      VL(1)=subs(P,x,m); VL(2)=subs(DP,x,m);
      VR(1)=subs(P,x,n); VR(2)=subs(DP,x,n);
 for i=3:length(Y)
    [R1,Q1] = polynomialReduce(P,DP);
    VL(i)=subs(-R1,x,m);VR(i)=subs(-R1,x,n); 
    d=coeffs(R1);
    if length(d)~=1
        P=DP; DP=-R1;
    else
        break
    end
 end
    fprintf('VS(a):[%s]\n', join(string(VL),','));
    fprintf('VS(b):[%s]\n',join(string(VR),','))
    VS1=nonzeros(VL); VS2=nonzeros(VR);
    fprintf('V_S(%2.2f) is %d\n',m,sign_chng(VS1))
    fprintf('V_S(%2.2f) is %d\n',n,sign_chng(VS2))
    w=sign_chng(VS1)-sign_chng(VS2);
   end
end
\end{lstlisting}
\begin{table}[!h]
\begin{center}
\begin{tabular}{ |p{2.5cm}|p{2.8cm}|p{2.8cm}|p{2cm}|p{2.5cm}|}
\hline
$F(\beta)$& $\beta=1$ &$\beta=3/2$&$\mathcal{N}_{q_n}(1,3/2]$& Sign of $F(\beta)$\\
 \hline 
 \multirow{5}{8em}{$P_\beta(1)=q_0(\beta)$}   & $q_0=3$    &$q_0=66.75$&\multirow{5}{2em}{$0$} & \multirow{5}{2em}{$>0$}   \\
 &$q_0^\prime=31$   & $q_0^\prime=287.5$ & &\\
&$q_0^{\prime\prime}=192$   & $q_0^{\prime\prime}=954$ & &\\
&$q_0^{\prime\prime\prime}=804$   & $q_0^{\prime\prime\prime}=2244$ & &\\
&$q_0^{(4)}=2880$   & $q_0^{(4)}=2880$ & &\\
 \hline
 \multirow{4}{8em}{$P^\prime_\beta(1)=24(\beta-1)q_1(\beta)$}   & $q_1=1$    &$q_1=5.25$&\multirow{4}{2em}{$0$} & \multirow{4}{2em}{$\ge0$}   \\
 &$q_1^\prime=4$   & $q_1^\prime=14.5$ & &\\
  &$q_1^{\prime\prime}=12$   & $q_1^{\prime\prime}=30$ & &\\
&$q_1^{\prime\prime\prime}=36$   & $q_1^{\prime\prime\prime}=36$ & &\\
 \hline
  \multirow{4}{8em}{$P^{\prime\prime}_\beta(1)=8(\beta-1)q_2(\beta)$}   & $q_2=1$    &$q_2=7.5$&\multirow{4}{2em}{$0$} & \multirow{4}{2em}{$\ge0$}   \\
 &$q_2^\prime=6$   & $q_2^\prime=23$ & &\\
  &$q_2^{\prime\prime}=16$   & $q_2^{\prime\prime}=52$ & &\\
&$q_2^{\prime\prime\prime}=72$   & $q_2^{\prime\prime\prime}=72$ & &\\
\hline
\multirow{4}{8em}{$P_\beta(-1)=(2-\beta)q_3(\beta)$}   & $q_3=3$    &$q_3=1.5$&\multirow{4}{2em}{$0$} & \multirow{4}{2em}{$>0$}   \\
 &$q_3^\prime=2$ & $q_3^\prime=-14$ & &\\
  &$q_3^{\prime\prime}=4$   & $q_3^{\prime\prime}=-68$ & &\\
&$q_3^{\prime\prime\prime}=-144$   & $q_3^{\prime\prime\prime}=-144$ & &\\
\hline
  \multirow{4}{8em}{$P^{\prime}_\beta(-1)=8(\beta-1)q_4(\beta)$}   & $q_4=1$    &$q_4=0.75$&\multirow{4}{2em}{$0$} & \multirow{4}{2em}{$\ge0$}   \\
 &$q_4^\prime=0$   & $q_4^\prime=-2.5$ & &\\
  &$q_4^{\prime\prime}=4$   & $q_4^{\prime\prime}=-14$ & &\\
&$q_4^{\prime\prime\prime}=-36$   & $q_4^{\prime\prime\prime}=-36$ & &\\
\hline
\end{tabular}
\caption{\label{case2}Sign of the sequence $\{P_\beta, P_\beta^\prime, P_\beta^{\prime\prime}\}$.}
\end{center}
\end{table}
\begin{table}[H]
\begin{tabular}{ |p{3cm}|p{2.5cm}|p{2.5cm}|p{2cm}|p{2.5cm}|}
\hline
$F(\beta)$& $\beta=3/2$ &$\beta=2$&$\mathcal{N}_{P_n}(3/2,2]$& Sign of $F(\beta)$\\
 \hline 
 \multirow{5}{8em}{$K_\beta(1)=4p_0(\beta)$}   & $p_0=16.6875$    &$p_0=192$&\multirow{5}{2em}{$0$} & \multirow{5}{2em}{$>0$}   \\
 &$p_0^\prime=134.75$   & $p_0^\prime=656$ & &\\
&$p_0^{\prime\prime}=558$   & $p_0^{\prime\prime}=1632$ & &\\
&$p_0^{\prime\prime\prime}=1518$   & $p_0^{\prime\prime\prime}=2778$ & &\\
&$p_0^{(4)}=2520$   & $p_0^{(4)}=2520$ & &\\
 \hline
 \multirow{5}{8em}{$K^\prime_\beta(1)=8p_1(\beta)$}   & $p_1=7.875$    &$p_1=192$&\multirow{5}{2em}{$0$} & \multirow{5}{2em}{$>0$}   \\
 &$p_1^\prime=128$   & $p_1^\prime=713$ & &\\
  &$p_1^{\prime\prime}=606$   & $p_1^{\prime\prime}=1860$ & &\\
&$p_1^{\prime\prime\prime}=1752$   & $p_1^{\prime\prime\prime}=3264$ & &\\
&$p_1^{(4)}=3024$   & $p_1^{(4)}=3024$ & &\\
 \hline
 \multirow{5}{8em}{$K^{\prime\prime}_\beta(1)=8p_2(\beta)$}   & $p_2=3.75$    &$p_2=288$&\multirow{5}{2em}{$0$} & \multirow{5}{2em}{$>0$}   \\
 &$p_2^\prime=170.5$   & $p_2^\prime=1141$ & &\\
  &$p_2^{\prime\prime}=990$   & $p_2^{\prime\prime}=3084$ & &\\
&$p_2^{\prime\prime\prime}=3036$   & $p_2^{\prime\prime\prime}=5340$ & &\\
&$p_2^{(4)}=4608$   & $p_2^{(4)}=4608$ & &\\
 \hline
  \multirow{3}{8em}{$K^{\prime\prime\prime}_\beta(1)=192(\beta-1)(2\beta-3)p_3(\beta)$}   & $p_3=5.75$    &$p_3=12$&\multirow{4}{2em}{$0$} & \multirow{4}{2em}{$\ge0$}   \\
 &$p_3^\prime=11$   & $p_3^\prime=14$ & &\\
  &$p_3^{\prime\prime}=6$   & $p_3^{\prime\prime}=6$ & &\\
\hline
\multirow{4}{8em}{$K_\beta^\prime(-1)=8(2-\beta)p_4(\beta)$}   & $p_4=0.75$    &$p_4=1$&\multirow{4}{2em}{$2$} & \multirow{4}{2em}{not able to determine}   \\
 &$p_4^\prime=-0.5$   & $p_4^\prime=-6$ & &\\
  &$p_4^{\prime\prime}=34$   & $p_4^{\prime\prime}=-56$ & &\\
&$p_4^{\prime\prime\prime}=-180$   & $p_4^{\prime\prime\prime}=-180$ & &\\
\hline
  \multirow{4}{8em}{$K^{\prime\prime}_\beta(-1)=8(2-\beta)p_5(\beta)$}   & $p_5=7.5$    &$p_5=11$&\multirow{4}{2em}{$2$} & \multirow{4}{2em}{not able to determine}   \\
 &$p_5^\prime=-4$   & $p_5^\prime=42$ & &\\
  &$p_5^{\prime\prime}=-52$   & $p_5^{\prime\prime}=236$ & &\\
&$p_5^{\prime\prime\prime}=576$   & $p_5^{\prime\prime\prime}=576$ & &\\
\hline
\end{tabular}
\caption{\label{bigtable}Sign of the sequence $\{K_\beta, K_\beta^\prime, K_\beta^{\prime\prime}, K_\beta^{\prime\prime\prime}\}$.}
\end{table}

\end{document}